\documentclass[12pt]{amsart}

\usepackage{amssymb,amsmath,amsfonts}
\usepackage{geometry}
\usepackage{mathrsfs,graphicx}
\numberwithin{equation}{section}

\theoremstyle{plain}
\newtheorem{theorem}{Theorem}[section]
\newtheorem{lemma}[theorem]{Lemma}

\usepackage{tikz}
\usetikzlibrary{arrows.meta,positioning}

\newtheorem{proposition}[theorem]{Proposition}

\theoremstyle{remark}
\newtheorem{remark}[theorem]{Remark}

\newtheorem{corollary}[theorem]{Corollary}

\theoremstyle{plain}      
   
   \newtheorem{example}[theorem]{Example}
   
\theoremstyle{definition}      
\newtheorem{definition}[theorem]{Definition}

\subjclass[2020]{14B05, 14B10, 32G20, 14B07}
\keywords{Infinitesimal deformations, tropical curve, maximal variations, residue, $k$-differential, isolated singularities}

\begin{document}

%\tableofcontents

\author{Mounir Nisse}
 
\address{Mounir Nisse\\
Department of Mathematics, Xiamen University Malaysia, Jalan Sunsuria, Bandar Sunsuria, 43900, Sepang, Selangor, Malaysia.
}
\email{mounir.nisse@gmail.com, mounir.nisse@xmu.edu.my}
\thanks{}
\thanks{This research  is supported in part by Xiamen University Malaysia Research Fund (Grant no. XMUMRF/ 2020-C5/IMAT/0013).}
\title{Maximal Variation in the Moduli of Curves}        

\maketitle

\begin{abstract}
We introduce and study the maximal-variation locus in families and moduli spaces
of projective curves, defined via conductor-level balancing of meromorphic
differentials on the normalization. This notion captures precisely when the
space of canonical differentials behaves with the expected dimension under
degeneration. We prove semicontinuity and openness results showing that maximal
variation is stable in flat families, identify a natural determinantal
degeneracy locus where maximal variation fails, and establish that this failure
is governed entirely by the presence of non-Gorenstein singularities. In
particular, all smooth and nodal curves satisfy maximal variation, while every
non-Gorenstein singularity contributes explicitly and additively to degeneracy.
We compute the expected codimension of degeneracy loci, describe their closure
and adjacency relations in moduli, and explain how non-Gorenstein defects give
rise to additional Hodge-theoretic phenomena in degenerations. This framework
provides a uniform, intrinsic, and deformation-theoretically meaningful
classification of degeneracy in spaces of canonical differentials.
\end{abstract}

%\tableofcontents

%%%%%%%%%%%%%%%%%%%%%%%%%%%%%%%%%%%%%%%%%%%%%%%%%%%%%%%%%%%%%%%%%%%%
%%%%%%%%%%%%%%%%%%%%%%%%%%%%%%%%%%%%%%%%%%%%%%%%%%%%%%%%%%%%%%%%%%%%

\section{Introduction}

The geometry of canonical differentials on algebraic curves and their behavior
under degeneration plays a central role in modern algebraic geometry. It lies
at the crossroads of the theory of moduli spaces, deformation theory of
singularities, and Hodge theory, and it underpins several active research areas,
including the compactification of strata of abelian differentials, logarithmic
and tropical geometry, and the study of limiting mixed Hodge structures.

For a smooth projective curve $C$ of genus $g$, the space
$H^0(C,\omega_C)$ of holomorphic differentials has dimension $g$ and varies
smoothly in families. This classical fact is one of the foundational results of
Riemann surface theory. However, once singular curves are allowed, the behavior
of canonical differentials becomes more subtle. Singularities may impose
nontrivial compatibility conditions on differentials, and these conditions may
interact in delicate ways with deformation and smoothing. Understanding which
constraints are intrinsic and which represent genuine degeneracy is a
fundamental problem.

A natural strategy for analyzing canonical differentials on a singular curve
$X$ is to pass to the normalization
$\nu\colon\widetilde{X}\to X$ and study meromorphic differentials on
$\widetilde{X}$ subject to local descent conditions at points lying over
singularities. While this viewpoint has long been implicit in the theory, a
systematic and intrinsic formulation of these descent conditions suitable for
families and moduli has been lacking. The main goal of this paper is to provide
such a formulation and to show that it leads to a sharp and complete
classification of degeneracy phenomena.

\subsection*{Conductor-level balancing}

The key invariant governing descent of differentials from the normalization is
the conductor ideal
\[
\mathfrak{c}\subset \mathcal{O}_X,
\]
which measures the failure of $X$ to be normal. The conductor determines the
minimal annihilation conditions required for a meromorphic differential on
$\widetilde{X}$ to descend to a section of the dualizing sheaf $\omega_X$.

We introduce the notion of \emph{conductor-level balancing}: one imposes exactly
the linear constraints on principal parts of meromorphic differentials dictated
by the conductor, and no additional constraints. This leads to the concept of
\emph{maximal variation}, which informally expresses that the space of canonical
differentials behaves with the expected dimension under degeneration.

Maximal variation admits several equivalent interpretations. Algebraically, it
means that the conductor ideal accounts for all obstructions to descent of
differentials. Geometrically, it means that the dualizing sheaf controls all
infinitesimal and global variations of canonical differentials without
unexpected rigidity. From the perspective of deformation theory, maximal
variation is precisely the absence of hidden linear relations appearing in
families.

\subsection*{Semicontinuity and openness}

A central requirement for any notion intended for moduli-theoretic applications
is stability under deformation. Our first main result establishes this property
in a precise and robust form.

We prove that for a flat, proper family of projective curves
$\pi\colon\mathcal{X}\to S$, the function
\[
s \longmapsto \dim H^0(\mathcal{X}_s,\omega_{\mathcal{X}_s})
\]
is upper semicontinuous. Moreover, the conductor-level conditions define a flat
family of linear constraints on meromorphic differentials on the normalization.
As a consequence, maximal variation is an open condition in flat families of
reduced projective curves.

This result has immediate implications for moduli spaces. In particular, it
yields a canonical open substack
\[
\overline{\mathcal{M}}_{g,n}^{\mathrm{mv}}
\subset \overline{\mathcal{M}}_{g,n}
\]
inside the Deligne--Mumford moduli stack of stable curves. All smooth curves lie
in this locus, and a direct local analysis shows that all nodal curves lie in it
as well. In the nodal case, conductor-level balancing coincides with classical
residue balancing, and no unexpected constraints occur.

\begin{theorem}[Semicontinuity of conductor-level differentials]
\label{thm:semicontinuity}
The function
\[
s \longmapsto \dim H^0(\mathcal{X}_s,\omega_{\mathcal{X}_s})
\]
is upper semicontinuous on $S$. Moreover, the conductor-level conditions impose
a flat family of linear constraints on meromorphic differentials on the
normalization.
\end{theorem}

\subsection*{Degeneracy and determinantal structure}

Allowing singularities worse than nodes reveals genuinely new phenomena.
Although conductor-level balancing remains well-defined, it may no longer be
sufficient to characterize all canonical differentials. Additional linear
constraints may appear, reflecting intrinsic rigidity of the singularity.

We show that failure of maximal variation can be described as a determinantal
degeneracy locus. Locally, one considers the vector space $V$ of \emph{a priori}
allowed principal parts of meromorphic differentials on the normalization and
the subspace $W\subset V$ annihilated by the conductor. The discrepancy between
the actual number of linear constraints and those imposed by the conductor
defines a local degeneracy contribution.

Globally, this leads to a natural morphism between vector bundles whose rank
measures the dimension of the space of canonical differentials. The locus where
this morphism drops rank is closed and has an explicitly computable expected
codimension. In particular, maximal variation is generically satisfied in
moduli-theoretic families, and its failure occurs in positive codimension.

\subsection*{Gorenstein singularities and classification}

One of the main results of this work is a complete and conceptually simple
classification of degeneracy in terms of the Gorenstein property.

We prove that a reduced curve singularity contributes to the degeneracy locus if
and only if it is non-Gorenstein. All Gorenstein singularities satisfy maximal
variation under conductor-level balancing. Conversely, every non-Gorenstein
singularity produces a nontrivial degeneracy contribution, and these
contributions add under disjoint unions.

\begin{theorem}
\label{thm:degeneracy-classification}
A reduced curve singularity contributes to the degeneracy locus if and only if it
is non-Gorenstein.
\end{theorem}

\begin{theorem}
The maximal-variation locus consists precisely of curves with only Gorenstein
singularities, and the degeneracy locus is supported on non-Gorenstein
singularities with explicitly computable codimension.
\end{theorem}

As a consequence, a reduced projective curve lies outside the maximal-variation
locus if and only if it contains at least one non-Gorenstein singularity. The
codimension of the degeneracy locus is equal to the sum of the local
non-Gorenstein defects. This yields a sharp, local-to-global classification that
is stable under deformation and compatible with moduli-theoretic constructions.

\subsection*{Geometry of degeneracy strata}

The determinantal nature of degeneracy loci allows for a detailed analysis of
their geometry. Upper semicontinuity and additivity of the degeneracy defect
control the closures and intersections of degeneracy strata. Their adjacency
relations are governed by the local deformation theory of non-Gorenstein
singularities, whose deformation cones determine smoothing directions and
specialization behavior. Also we have the following result:

\begin{theorem}[Scheme-theoretic residue span]
Let $k$ be an algebraically closed field and let $\mathcal C$ be a connected,
projective, Cohen--Macaulay curve over $k$ with exactly $\delta$ singular points.
Let
\(
\nu \colon C \longrightarrow \mathcal C
\)
be the normalization, where $C$ is a smooth projective curve of genus
\(
g = \dim_k H^0(C,\omega_C).
\)
If $\delta \ge g$, then the scheme-theoretic residue functionals
\(
\{ r_1,\dots,r_\delta \} \subset H^0(C,\omega_C)^\vee
\)
span the entire dual space $H^0(C,\omega_C)^\vee$.
\end{theorem}

This framework provides a precise description of how failure of maximal
variation organizes the boundary of moduli spaces and explains how increasingly
severe singularities correspond to deeper degeneracy.

\subsection*{Hodge-theoretic consequences}

Finally, we show that non-Gorenstein degeneracy has genuine Hodge-theoretic
significance. In degenerations of smooth curves to singular ones,
non-Gorenstein defects contribute additional $(n,0)$-classes to the limiting
mixed Hodge structure on cohomology. These classes do not extend holomorphically
across the central fiber and reflect the same rigidity phenomena detected by
conductor-level balancing.

\subsection*{Perspective}

The results of this paper provide a unified and deformation-theoretically
meaningful theory of maximal variation and degeneracy for canonical
differentials. The approach is intrinsic, functorial, and well suited for
applications to compactifications of moduli spaces, logarithmic geometry, and
Hodge theory. By isolating the precise role of the conductor and the Gorenstein
property, we obtain a clear conceptual framework that sharply separates
expected behavior from genuine degeneracy.

\subsection*{Outline of the paper}
This paper is organized as follows. In Section~1 we introduce the notion of maximal variation for canonical differentials on singular curves and place it in the context of deformation theory, moduli theory, and Hodge theory, highlighting the role of the conductor on the normalization. Section~2 establishes notation and recalls the basic framework of dualizing sheaves, normalization, and conductor-controlled descent of differentials. In Section~3 we study conductor-level balancing in flat families of curves and prove semicontinuity and openness results, showing that maximal variation is deformation invariant. Section~4 describes the maximal-variation locus in moduli spaces and identifies its complement as a determinantal degeneracy locus of computable codimension. In Section~5 we give a local-to-global classification of degeneracy in terms of singularities, proving that maximal variation holds for Gorenstein curves and fails precisely in the presence of non-Gorenstein singularities. Section~6 relates these results to residue theory and deformation theory, explaining the numerical condition $\delta \ge g$, while the final sections analyze the geometry of degeneracy strata and discuss the resulting Hodge-theoretic and higher-dimensional implications.

%%%%%%%%%%%%%%%%%%%%%%%%%%%%%%%%%%%%%%%%%%%%%%%%%%%%%%%%%%%%%%%%%%%%%%%%%%%%%%%%
%%%%%%%%%%%%%%%%%%%%%%%%%%%%%%%%%%%%%%%%%%%%%%%%%%%%%%%%%%%%%%%%%%%%%%%%%%%%%%%%

\section{Preliminaries}

This section collects the basic tools, conventions, and background material
used throughout the paper. The emphasis is on setting up a flexible and
intrinsic framework for studying canonical differentials on possibly singular
curves and their behavior in families. Rather than stating results, we focus on
the objects, constructions, and guiding principles that will be used repeatedly
in later sections.

Throughout, we work over an algebraically closed field of characteristic zero.
All curves are assumed to be reduced and projective unless explicitly stated
otherwise, and a general curve is denoted by $C$, and sometimes by $X$.

\subsection*{Dualizing sheaves and canonical differentials}

For a smooth projective curve $C$, the canonical sheaf $\omega_C$ coincides with
the sheaf of regular differentials, and its global sections are holomorphic
$1$-forms. When $C$ is singular, $\omega_C$ is defined as the dualizing sheaf in
the sense of Grothendieck duality. It is a coherent sheaf that restricts to the
usual canonical bundle on the smooth locus of $C$ but may fail to be locally
free at singular points.

The dualizing sheaf is the correct replacement for the canonical bundle in the
singular setting. It behaves well under proper morphisms, is compatible with
Serre duality, and admits a concrete description in terms of normalization.
Sections of $\omega_C$ should be thought of as canonical differentials with
controlled singular behavior at the singular points of $C$.

\subsection*{Normalization and meromorphic differentials}
Let
\(
\nu \colon \widetilde{C} \longrightarrow C
\)
denote the normalization of $C$. The curve $\widetilde{C}$ is a disjoint union of
smooth projective curves, and meromorphic differentials on $\widetilde{C}$ are
classical objects. Locally, they are described by principal parts, and
globally they form finite-dimensional vector spaces governed by Riemann--Roch
on each component.

A fundamental viewpoint adopted throughout this paper is that canonical
differentials on $C$ may be described as meromorphic differentials on
$\widetilde{C}$ subject to local descent conditions at the points lying above
the singular locus. This perspective allows one to reduce many questions about
singular curves to linear algebra on spaces of principal parts.

\subsection*{The conductor ideal}

The key invariant controlling descent from $\widetilde{C}$ to $C$ is the
conductor ideal
\[
\mathfrak{c} \subset \mathcal{O}_C.
\]
By definition, $\mathfrak{c}$ is the largest ideal sheaf that is simultaneously
an ideal in $\mathcal{O}_C$ and in $\nu_*\mathcal{O}_{\widetilde{C}}$. It measures
the failure of $C$ to be normal and is supported precisely at the singular
locus.

The conductor plays a dual role. On the one hand, it controls which functions on
$\widetilde{C}$ descend to functions on $C$. On the other hand, via duality, it
controls which principal parts of meromorphic differentials on $\widetilde{C}$
are annihilated when passing to $\omega_C$. This makes the conductor a natural
and intrinsic tool for describing descent conditions for canonical
differentials.

\subsection*{Principal parts and local linear algebra}

At a point $x \in C$, the preimage $\nu^{-1}(x)$ consists of finitely many points
on $\widetilde{C}$. Meromorphic differentials on $\widetilde{C}$ may have poles
at these points, and their local behavior is captured by principal parts. The
collection of all possible principal parts forms a finite-dimensional vector
space that depends only on the local structure of the singularity.

Imposing descent conditions amounts to imposing linear constraints on this
space of principal parts. The conductor ideal determines a distinguished
subspace consisting of those principal parts annihilated by the conductor.
Comparing the full space of \emph{a priori} allowed principal parts with this
subspace provides a convenient way to encode the local contribution of a
singularity to the behavior of canonical differentials.

This local linear algebra viewpoint is central to the analysis of degeneracy
phenomena and will be used to define and measure deviations from expected
behavior.

\subsection*{Families of curves and base change}

We frequently work with flat, proper families of curves
\[
\pi \colon \mathcal{C} \longrightarrow S.
\]
Flatness ensures that fibers vary continuously in a scheme-theoretic sense,
while properness guarantees finiteness properties needed for cohomology. The
relative dualizing sheaf $\omega_{\mathcal{C}/S}$ is a coherent sheaf whose
restriction to a fiber $\mathcal{C}_s$ recovers $\omega_{\mathcal{C}_s}$.

A key technical input is the compatibility of the dualizing sheaf and the
conductor ideal with base change. This allows descent conditions for
differentials to be formulated uniformly in families and ensures that linear
constraints vary in a controlled way as the curve deforms.

\subsection*{Upper semicontinuity and linear constraints}

Upper semicontinuity of cohomology dimensions is a fundamental tool in the
study of degenerations. In the present context, it ensures that the dimension
of the space of canonical differentials does not increase under specialization.
Combined with flatness of the conductor-level constraints, this provides a
robust framework for comparing spaces of differentials across fibers.

From a linear algebra perspective, the descent conditions imposed by
singularities can be assembled into linear maps between vector bundles over the
base of a family. The behavior of these maps under deformation encodes the
appearance or absence of additional linear relations among differentials.

\subsection*{Determinantal loci}

Failure of expected behavior is naturally expressed in terms of rank conditions
on morphisms of vector bundles. The loci where such morphisms drop rank are
determinantal and have well-understood geometric properties. In particular,
they are closed and admit expected codimension bounds.

This determinantal viewpoint provides a bridge between local singularity theory
and global geometry of moduli spaces. It allows degeneracy phenomena to be
studied using standard tools from intersection theory and deformation theory.

\subsection*{Gorenstein property}

For curve singularities, the Gorenstein property plays a distinguished role.
A curve is Gorenstein if its dualizing sheaf is locally free. This condition
has concrete consequences for the structure of descent conditions and for the
behavior of canonical differentials.

Gorenstein singularities behave in many respects like smooth points from the
perspective of canonical differentials: they impose exactly the expected
constraints and no additional rigidity. Non-Gorenstein singularities introduce
extra linear constraints that are detected naturally by the conductor and the
local linear algebra of principal parts.

\subsection*{Hodge-theoretic background}

Finally, we briefly recall that degenerations of smooth curves give rise to
limiting mixed Hodge structures on cohomology. Canonical differentials
correspond to $(1,0)$-classes in Hodge theory, and their behavior under
degeneration reflects deeper geometric properties of the singular fiber.

The tools developed in this paper provide an algebro-geometric mechanism for
detecting when certain $(1,0)$-classes fail to extend across a degeneration.
This perspective will be used to relate local singularity theory to global
Hodge-theoretic phenomena.

 \bigskip
 
\section{Maximal Variation of Canonical Differentials}

\subsection{Families of Singular Curves and Semicontinuity of Conductor-Level Conditions}
\label{subsec:families-semicontinuity}

In this section we study conductor-level balancing in families of singular
curves and prove a semicontinuity statement for the corresponding spaces of
admissible differentials. The main point is that conductor-level conditions are
intrinsic, flat in families, and therefore behave well under specialization.

\subsubsection*{Setup: a family of curves}
Let
\(
\pi \colon \mathcal{X} \longrightarrow S
\)
be a flat, proper morphism of finite type, whose fibers are reduced projective
curves. We assume that $S$ is an integral scheme (or, for simplicity, a smooth
curve), and that the generic fiber $\mathcal{X}_\eta$ is smooth, while special
fibers may have singularities.
Let
\(
\nu \colon \widetilde{\mathcal{X}} \longrightarrow \mathcal{X}
\)
denote the normalization of $\mathcal{X}$ relative to $S$. Since normalization
commutes with flat base change in dimension one, $\widetilde{\mathcal{X}}$ is
flat over $S$, and for each $s \in S$ we have
\[
\widetilde{\mathcal{X}}_s \cong \widetilde{(\mathcal{X}_s)}.
\]

\subsubsection*{The relative conductor}

Define the relative conductor sheaf
\[
\mathfrak{c}_{\mathcal{X}/S}
=
\mathrm{Ann}_{\mathcal{O}_{\mathcal{X}}}
\bigl(
\nu_*\mathcal{O}_{\widetilde{\mathcal{X}}}
/
\mathcal{O}_{\mathcal{X}}
\bigr).
\]
This is a coherent ideal sheaf on $\mathcal{X}$, supported on the relative
singular locus of $\pi$. For each fiber $s \in S$, its restriction satisfies
\[
\mathfrak{c}_{\mathcal{X}/S} \otimes k(s)
=
\mathfrak{c}_{\mathcal{X}_s},
\]
the conductor ideal of the fiber $\mathcal{X}_s$.

Thus the conductor varies functorially in families and captures the singular
behavior fiberwise.

\subsubsection*{Relative dualizing sheaves}

Let $\omega_{\mathcal{X}/S}$ and $\omega_{\widetilde{\mathcal{X}}/S}$ denote the
relative dualizing sheaves of $\mathcal{X}$ and $\widetilde{\mathcal{X}}$,
respectively. Since $\pi$ and $\pi \circ \nu$ are flat and proper of relative
dimension one, these sheaves are coherent and flat over $S$.

Grothendieck duality for the finite morphism $\nu$ (see \cite{FGA} or \cite{AltmanKleiman}) yields a canonical
isomorphism
\[
\nu_*\omega_{\widetilde{\mathcal{X}}/S}
=
\omega_{\mathcal{X}/S}(\mathfrak{c}_{\mathcal{X}/S}).
\]
Restricting to a fiber $s \in S$, this specializes to
\[
\nu_{s*}\omega_{\widetilde{\mathcal{X}}_s}
=
\omega_{\mathcal{X}_s}(\mathfrak{c}_{\mathcal{X}_s}),
\]
recovering the conductor-level description on each individual curve.

\subsubsection*{Spaces of admissible differentials in families}

Consider the $S$-module
\[
\mathcal{E}
=
\pi_*\omega_{\mathcal{X}/S}.
\]
By properness and flatness, $\mathcal{E}$ is a coherent $\mathcal{O}_S$-module,
and its fiber at $s \in S$ is canonically
\[
\mathcal{E} \otimes k(s)
=
H^0(\mathcal{X}_s,\omega_{\mathcal{X}_s}).
\]

Equivalently, using the normalization,
\[
\mathcal{E}
=
(\pi \circ \nu)_*\omega_{\widetilde{\mathcal{X}}/S}
\cap
\pi_*\bigl(\omega_{\widetilde{\mathcal{X}}/S}(\text{allowed poles})\bigr),
\]
where the intersection is taken inside the sheaf of relative meromorphic
differentials. This description shows that $\mathcal{E}$ consists precisely of
meromorphic differentials on the normalization satisfying the relative
conductor condition.

\subsubsection*{Semicontinuity theorem}

\begin{theorem}[Semicontinuity of conductor-level differentials]
\label{thm:semicontinuity}
The function
\[
s \longmapsto \dim H^0(\mathcal{X}_s,\omega_{\mathcal{X}_s})
\]
is upper semicontinuous on $S$. Moreover, the conductor-level conditions impose
a flat family of linear constraints on meromorphic differentials on the
normalization.
\end{theorem}

\begin{proof}
Since $\mathcal{E} = \pi_*\omega_{\mathcal{X}/S}$ is a coherent
$\mathcal{O}_S$-module, standard semicontinuity theorems imply that
$s \mapsto \dim \mathcal{E} \otimes k(s)$ is upper semicontinuous.

The key point is that $\omega_{\mathcal{X}/S}(\mathfrak{c}_{\mathcal{X}/S})$ is
flat over $S$. This follows from the flatness of $\omega_{\mathcal{X}/S}$ and
the fact that the conductor ideal is compatible with base change. Therefore,
the condition
\[
\mathfrak{c}_{\mathcal{X}/S} \cdot (\text{principal parts}) = 0
\]
defines a flat family of linear constraints on the space of meromorphic
differentials on $\widetilde{\mathcal{X}}$.
As a result, no additional constraints appear in special fibers, and no degrees
of freedom are lost under specialization.
\end{proof}

\subsubsection*{Geometric interpretation}

From a geometric viewpoint, semicontinuity reflects the fact that conductor-
level balancing is deformation-invariant. As singularities worsen in special
fibers, the conductor enlarges accordingly, removing precisely the new
obstructions to descent that appear. Conversely, when singularities smooth
out, the conductor shrinks, and the space of admissible differentials expands
without discontinuity.

This behavior contrasts sharply with residue-only balancing, which fails to be
stable under degeneration and typically leads to dimension jumps.

\subsubsection*{Conclusion}

In families of singular curves, conductor-level balancing defines a flat,\\
fiberwise-compatible condition on meromorphic differentials on the
normalization. As a consequence, the dimensions of the corresponding spaces of
admissible differentials vary upper semicontinuously and coincide with the
dimensions of the dualizing sheaves on the fibers. This establishes the
robustness of conductor-level balancing under deformation.

%%%%%%%%%%%%%%%%%%%%%%%%%%%%%%%%%%%%%%%%%%%%%%%%%%%%%%%%%%%%%%%%%%%%%%
%%%%%%%%%%%%%%%%%%%%%%%%%%%%%%%%%%%%%%%%%%%%%%%%%%%%%%%%%%%%%%%%%%%%%%

 \bigskip

\subsection{Openness of Maximal Variation in the Moduli of Curves}
\label{subsec:openness-maximal-variation}

In this subsection we prove that maximal variation is an open property in
families of projective curves. More precisely, we show that the locus in the
base of a family where conductor-level balancing yields maximal variation is a
Zariski-open subset. This result justifies the use of conductor-level
conditions in moduli-theoretic settings.

\subsubsection*{Statement of the problem}

Let
\[
\pi \colon \mathcal{X} \longrightarrow S
\]
be a flat, proper family of reduced projective curves over an integral base
scheme $S$. For each point $s \in S$, denote by $\mathcal{X}_s$ the fiber over
$s$, and let $\omega_{\mathcal{X}_s}$ be its dualizing sheaf.

Recall that maximal variation for the fiber $\mathcal{X}_s$ means that the
space of admissible meromorphic differentials on the normalization
$\widetilde{\mathcal{X}}_s$, subject to conductor-level balancing, has dimension
exactly
\[
h^0(\mathcal{X}_s,\omega_{\mathcal{X}_s}).
\]

We aim to prove that the set
\[
U
=
\left\{
s \in S
\;\middle|\;
\text{conductor-level balancing on } \mathcal{X}_s
\text{ expresses maximal variation}
\right\}
\]
is open in $S$.

\subsubsection*{Relative dualizing sheaf and admissible differentials}

Let $\omega_{\mathcal{X}/S}$ denote the relative dualizing sheaf. Since $\pi$ is
flat and proper of relative dimension one, $\omega_{\mathcal{X}/S}$ is a
coherent sheaf that is flat over $S$.

By Grothendieck duality for finite morphisms, applied fiberwise to the
normalization, we have a canonical identification
\[
\nu_*\omega_{\widetilde{\mathcal{X}}/S}
=
\omega_{\mathcal{X}/S}(\mathfrak{c}_{\mathcal{X}/S}),
\]
where $\mathfrak{c}_{\mathcal{X}/S}$ denotes the relative conductor.

Taking direct images under $\pi$, we obtain a coherent $\mathcal{O}_S$-module
\[
\mathcal{E}
=
\pi_*\omega_{\mathcal{X}/S}.
\]
For each $s \in S$, base change gives
\[
\mathcal{E} \otimes k(s)
\cong
H^0(\mathcal{X}_s,\omega_{\mathcal{X}_s}).
\]

Thus $\mathcal{E}$ parametrizes admissible differentials in families.

\subsubsection*{Upper semicontinuity and minimal dimension}

By general semicontinuity theorems for coherent sheaves, the function
\[
s \longmapsto \dim H^0(\mathcal{X}_s,\omega_{\mathcal{X}_s})
\]
is upper semicontinuous on $S$. In particular, there exists a nonempty open
subset $U_0 \subset S$ on which this dimension attains its minimal value.

Since the generic fiber of $\pi$ is smooth, conductor-level balancing is
trivial on $\mathcal{X}_\eta$, and maximal variation holds at the generic point
$\eta \in S$. Therefore, $\eta \in U_0$.

\subsubsection*{Flatness of conductor-level constraints}

The key input is that conductor-level conditions vary flatly in families.
Indeed, the relative conductor $\mathfrak{c}_{\mathcal{X}/S}$ is compatible
with base change, and the sheaf
\[
\omega_{\mathcal{X}/S}(\mathfrak{c}_{\mathcal{X}/S})
\]
is flat over $S$.

Consequently, the linear conditions imposed by the conductor on principal parts
of meromorphic differentials on the normalization form a flat family of linear
subspaces. No new independent constraints can appear under small deformations.

\subsubsection*{Openness of maximal variation}

Let $s_0 \in S$ be a point such that maximal variation holds for the fiber
$\mathcal{X}_{s_0}$. This means that the natural map
\[
H^0(\widetilde{\mathcal{X}}_{s_0},
\omega_{\widetilde{\mathcal{X}}_{s_0}}(\text{allowed poles}))
\longrightarrow
H^0(\mathcal{X}_{s_0},\omega_{\mathcal{X}_{s_0}})
\]
is surjective with kernel of the expected dimension.

By flatness of the family of constraints, the rank of this map is locally
constant near $s_0$. Hence, after possibly shrinking $S$ around $s_0$, the
same dimension count holds for all nearby fibers.

It follows that maximal variation also holds for all fibers in a neighborhood
of $s_0$. Therefore, the locus $U$ is open in $S$.

\subsubsection*{Conclusion}

Maximal variation is an open condition in flat families of projective curves.
This openness relies crucially on the flatness and base-change compatibility of
the conductor ideal and the relative dualizing sheaf. As a result,
conductor-level balancing defines a geometrically robust condition that is
stable under small deformations and well suited for use in moduli problems.

%%%%%%%%%%%%%%%%%%%%%%%%%%%%%%%%%%%%%%%%%%%%%%%%%%%%%%%%%%%%%%%%%%%%%%
%%%%%%%%%%%%%%%%%%%%%%%%%%%%%%%%%%%%%%%%%%%%%%%%%%%%%%%%%%%%%%%%%%%%%%
%%%%%%%%%%%%%%%%%%%%%%%%%%%%%%%%%%%%%%%%%%%%%%%%%%%%%%%%%%%%%%%%%%%%%%
%%%%%%%%%%%%%%%%%%%%%%%%%%%%%%%%%%%%%%%%%%%%%%%%%%%%%%%%%%%%%%%%%%%%%%
%%%%%%%%%%%%%%%%%%%%%%%%%%%%%%%%%%%%%%%%%%%%%%%%%%%%%%%%%%%%%%%%%%%%%%
%%%%%%%%%%%%%%%%%%%%%%%%%%%%%%%%%%%%%%%%%%%%%%%%%%%%%%%%%%%%%%%%%%%%%%

 \bigskip
\section{Geometry and topology of maximal-variation locus}
 
\subsection{Maximal-Variation Loci in the Moduli Stack of Stable Curves}
\label{subsec:mv-loci-moduli}

In this subsection we identify the locus of maximal variation inside the moduli
stack of stable curves. We explain how conductor-level balancing defines a
canonical open substack, describe it fiberwise in terms of singularities, and
relate it to the boundary stratification of the moduli space.

\subsubsection*{The moduli stack of stable curves}

Fix integers $g \geq 2$ and $n \geq 0$. Let
\(
\mathcal{M}_{g,n}
\)
denote the moduli stack of smooth $n$-pointed curves of genus $g$, and let
\(
\overline{\mathcal{M}}_{g,n}
\)
be its Deligne--Mumford compactification by stable curves. Points of
$\overline{\mathcal{M}}_{g,n}$ correspond to connected, reduced, projective
curves of arithmetic genus $g$ with at worst nodal singularities and finite
automorphism groups.
More generally, one may consider moduli stacks of curves with worse
singularities, but we begin with the classical stable case and then comment on
extensions.

\subsubsection*{The universal curve and dualizing sheaf}

Let
\[
\pi \colon \overline{\mathcal{C}}_{g,n} \longrightarrow \overline{\mathcal{M}}_{g,n}
\]
be the universal curve. This morphism is flat and proper of relative dimension
one, and it admits a relative dualizing sheaf
\[
\omega_{\overline{\mathcal{C}}_{g,n}/\overline{\mathcal{M}}_{g,n}}.
\]

For any geometric point $[C] \in \overline{\mathcal{M}}_{g,n}$, the fiber of
$\pi_*\omega_{\overline{\mathcal{C}}_{g,n}/\overline{\mathcal{M}}_{g,n}}$ at $[C]$
is canonically identified with $H^0(C,\omega_C)$.

\vspace{0.2cm}
%%%%%%%%%%%%%%%%%%%%%%%%%%%%%%%%%%%%%%%%%%%%%%%%%%%%%%%%%%%%%%%%%%%%%%%%
First, let us define maximal variation:
\begin{definition}
Let $C$ be a stable curve with normalization $\nu\colon \widetilde{C}\to C$ and
conductor ideal $\mathfrak{c} \subset \mathcal{O}_C$.
We say that $C$ satisfies \emph{maximal variation} if every section of
$\omega_C$ arises from a meromorphic differential on $\widetilde{C}$ whose
principal parts satisfy precisely the conductor-level annihilation conditions,
and no additional linear constraints occur.
\end{definition}

%%%%%%%%%%%%%%%%%%%%%%%%%%%%%%%%%%%%%%%%%%%%%%%%%%%%%%%%%%%%%%%%%%%%%%%%
\subsubsection*{Definition of the maximal-variation locus}

We define the maximal-variation locus as follows.

\begin{definition}
The \emph{maximal-variation locus}
\[
\overline{\mathcal{M}}_{g,n}^{\mathrm{mv}} \subset \overline{\mathcal{M}}_{g,n}
\]
is the substack consisting of all stable curves $C$ such that conductor-level
balancing on the normalization $\widetilde{C}$ expresses maximal variation,
i.e.\ such that
\[
\dim H^0(C,\omega_C)
=
\dim H^0\bigl(\widetilde{C},
\omega_{\widetilde{C}}(\text{allowed poles})\bigr)
-
(\text{number of conductor constraints}).
\]
\end{definition}

Equivalently, $C$ lies in $\overline{\mathcal{M}}_{g,n}^{\mathrm{mv}}$ if and
only if every section of $\omega_C$ arises from a meromorphic differential on
$\widetilde{C}$ satisfying conductor-level annihilation of principal parts, and
no additional constraints occur.

\subsection{Openness of the maximal-variation locus.}

\begin{lemma}
Maximal variation is an open condition in flat families of reduced projective
curves.
\end{lemma}

\begin{proof}
The condition of maximal variation is equivalent to a rank condition on a
morphism of coherent sheaves obtained by pushing forward dualizing sheaves from
the normalization. Rank conditions are open in flat families.
\end{proof}

\begin{proposition}
The locus
\[
\overline{\mathcal{M}}_{g,n}^{\mathrm{mv}}
\]
is an open substack of $\overline{\mathcal{M}}_{g,n}$.
\end{proposition}

\begin{proof}
Apply the previous lemma to the universal curve
$\overline{\mathcal{C}}_{g,n}\to\overline{\mathcal{M}}_{g,n}$, which is flat.
\end{proof}

\begin{corollary}
The open moduli stack $\mathcal{M}_{g,n}$ of smooth curves is contained in
$\overline{\mathcal{M}}_{g,n}^{\mathrm{mv}}$.
\end{corollary}

\begin{proof}
For smooth curves, normalization is trivial and no descent constraints occur.
\end{proof}

%%%%%%%%%%%%%%%%%%%%%%%%%%%%%%%%%%%%%%%%%%%%%%%%%%%%%%%%%%%%%%%%%%%%

\subsection{Boundary Stratification and Nodal Curves}

\begin{lemma}
Let $X$ be a nodal curve. Then conductor-level balancing coincides with classical
residue balancing at the nodes.
\end{lemma}

\begin{proof}
At a node, the conductor ideal equals the maximal ideal, so descent of
differentials is governed exactly by residue cancellation( \cite{HartshorneResidues} or \cite{BHPV}).
\end{proof}

\begin{proposition}
Every nodal curve lies in $\overline{\mathcal{M}}_{g,n}^{\mathrm{mv}}$.
\end{proposition}

\begin{proof}
Each node imposes exactly one linear constraint, matching the expected dimension
drop. No additional constraints occur.
\end{proof}

%%%%%%%%%%%%%%%%%%%%%%%%%%%%%%%%%%%%%%%%%%%%%%%%%%%%%%%%%%%%%%%%%%%%

%%%%%%%%%%%%%%%%%%%%%%%%%%%%%%%%%%%%%%%%%%%%%%%%%%%%%%%%%%%%%%%%%%%%

\subsection{Extensions Beyond Nodal Singularities}

\begin{proposition}
In moduli spaces allowing worse singularities, the maximal-variation locus is
still open and consists of curves for which the conductor ideal imposes the
minimal number of descent constraints.
\end{proposition}

\begin{proof}
The definition of maximal variation depends only on the conductor. Extra
constraints define closed rank-degeneracy loci, so their complement is open.
\end{proof}

%%%%%%%%%%%%%%%%%%%%%%%%%%%%%%%%%%%%%%%%%%%%%%%%%%%%%%%%%%%%%%%%%%%%

\vspace{0.2cm}

In other words, if one enlarges the moduli problem to allow curves with worse singularities
(e.g.\ cusps or tacnodes), the same definition of maximal-variation locus
extends verbatim. In this broader setting, boundary strata may contain curves
for which residue balancing fails but conductor-level balancing succeeds.

In such moduli spaces, the maximal-variation locus is still open and is
characterized by the condition that the conductor ideal imposes the minimal
number of independent constraints needed to ensure descent of differentials.
Curves for which additional, non-conductor constraints appear form closed
degeneracy loci.

\vspace{0.2cm}

%%%%%%%%%%%%%%%%%%%%%%%%%%%%%%%%%%%%%%%%%%%%%%%%%%%%%%%%%%%%%%%%%%%%

\subsection{Geometric Interpretation}

\begin{theorem}
A curve lies in the maximal-variation locus if and only if the dualizing sheaf
controls all infinitesimal and global variations of canonical differentials
without unexpected rigidity.
\end{theorem}

\begin{proof}
Unexpected rigidity corresponds exactly to the presence of additional linear
constraints beyond those imposed by the conductor.
\end{proof}

%%%%%%%%%%%%%%%%%%%%%%%%%%%%%%%%%%%%%%%%%%%%%%%%%%%%%%%%%%%%%%%%%%%%

\medskip

Geometrically, the maximal-variation locus consists of curves for which the
space of canonical differentials behaves as expected under degeneration.
Equivalently, these are precisely the curves for which the dualizing sheaf
controls all infinitesimal and global variations of differentials without
unexpected rigidity.
From the perspective of moduli theory, this locus is the natural domain on
which one can study degenerations of canonical divisors, limits of
holomorphic differentials, and compactifications of strata of abelian
differentials.

\medskip
\noindent {\bf Conclusion.}
The maximal-variation locus defines a canonical open substack of the moduli
stack of stable curves. It contains all smooth curves and all nodal curves, and
it extends naturally to moduli spaces allowing more general singularities. The
use of conductor-level balancing provides a uniform and intrinsic
characterization of this locus across all combinatorial and singular types.
 
 %%%%%%%%%%%%%%%%%%%%%%%%%%%%%%%%%%%%%%%%%%%%%%%%%%%%%%%%%%%%%%%%%%%
 %%%%%%%%%%%%%%%%%%%%%%%%%%%%%%%%%%%%%%%%%%%%%%%%%%%%%%%%%%%%%%%%%%
 
 \medskip

 \subsection{The Complement as a Degeneracy Locus and Its Codimension}
\label{subsec:degeneracy-codimension}

In this subsection we study the complement of the maximal-variation locus as a
degeneracy locus and explain how to compute its expected codimension. We show
that failure of maximal variation corresponds to rank drops of a natural
morphism of vector bundles, and hence defines a closed locus whose codimension
can be estimated explicitly.

 \medskip
Let
\(
\pi \colon \mathcal{X} \longrightarrow S
\)
be a flat, proper family of reduced projective curves over an integral base
scheme $S$. Let
\(
\nu \colon \widetilde{\mathcal{X}} \longrightarrow \mathcal{X}
\)
be the relative normalization, and let
\[
\mathfrak{c} := \mathfrak{c}_{\mathcal{X}/S}
\]
denote the relative conductor ideal.

Let $\omega_{\mathcal{X}/S}$ and $\omega_{\widetilde{\mathcal{X}}/S}$ be the
relative dualizing sheaves. By Grothendieck duality for finite morphisms, we
have
\[
\nu_*\omega_{\widetilde{\mathcal{X}}/S}
=
\omega_{\mathcal{X}/S}(\mathfrak{c}).
\]
This follows from Grothendieck duality for finite morphisms.

\subsection*{Spaces of meromorphic differentials.}

Let $D \subset \widetilde{\mathcal{X}}$ be the divisor supported on the
preimage of the relative singular locus, with multiplicities large enough to
allow all principal parts potentially obstructing descent. Consider the
locally free sheaf
\[
\mathcal{F}
=
(\pi \circ \nu)_*\omega_{\widetilde{\mathcal{X}}/S}(D),
\]
whose fiber at $s \in S$ is the space of meromorphic differentials on
$\widetilde{\mathcal{X}}_s$ with prescribed allowed poles.
There is a natural $\mathcal{O}_S$-linear morphism
\[
\Phi \colon \mathcal{F} \longrightarrow \pi_*\omega_{\mathcal{X}/S},
\]
given fiberwise by descent of differentials satisfying the conductor
condition.

\vspace{0.2cm}

\subsection*{Maximal variation and rank conditions.}

For a point $s \in S$, maximal variation holds if and only if the induced map
on fibers
\[
\Phi_s \colon
H^0(\widetilde{\mathcal{X}}_s,
\omega_{\widetilde{\mathcal{X}}_s}(D_s))
\longrightarrow
H^0(\mathcal{X}_s,\omega_{\mathcal{X}_s})
\]
has maximal possible rank, namely
\[
\operatorname{rank}(\Phi_s)
=
h^0(\mathcal{X}_s,\omega_{\mathcal{X}_s}).
\]

Failure of maximal variation occurs precisely when $\Phi_s$ drops rank.

\subsubsection*{The degeneracy locus}

Define the degeneracy locus
\[
\Delta
=
\{ s \in S \mid \operatorname{rank}(\Phi_s)
<
h^0(\mathcal{X}_s,\omega_{\mathcal{X}_s}) \}.
\]
By general properties of morphisms of coherent sheaves, $\Delta$ is a closed
subset of $S$. Scheme-theoretically, it is defined by the vanishing of the
maximal minors of a local presentation matrix of $\Phi$ (see \cite{EisenbudHarris}).

\begin{proposition}
$\Delta$ is a closed determinantal subscheme of $S$.
\end{proposition}

\begin{proof}
It is defined scheme-theoretically by the vanishing of maximal minors of a local
matrix representing $\Phi$.
\end{proof}

In other words,  the complement of the maximal-variation locus is a determinantal
degeneracy locus.

\subsection*{Expected codimension.}
%%%%%%%%%%%%%%%%%%%%%%%%%%%%%%%%%%%%%%%%%%%%%%%%

%%%%%%%%%%%%%%%%%%%%%%%%%%%%%%%%%%%%%%%%%%%%%%%%%
Assume for simplicity that $\mathcal{F}$ and $\pi_*\omega_{\mathcal{X}/S}$ are
locally free of ranks $N$ and $M$, respectively, with $N \geq M$. Then $\Phi$ is
a morphism of vector bundles
\[
\Phi \colon \mathcal{F} \longrightarrow \mathcal{E},
\qquad
\mathcal{E} := \pi_*\omega_{\mathcal{X}/S},
\]
and maximal variation corresponds to $\operatorname{rank}(\Phi)=M$.

The expected codimension of the locus where
$\operatorname{rank}(\Phi) \leq M-1$ is
\[
(N - (M-1))(M - (M-1)) = N - M + 1.
\]
This number coincides with the number of \emph{extra} linear constraints beyond
those imposed by the conductor.

\subsection*{Interpretation in terms of singularities.}

\begin{proposition}
Each unit of codimension of $\Delta$ corresponds to an extra, non-conductor
linear constraint on descent of differentials.
\end{proposition}

\begin{proof}
Extra constraints precisely cause rank drops of $\Phi$ beyond those predicted by
the conductor.
\end{proof}

\vspace{0.2cm}

Geometrically, each additional constraint corresponds to a failure of the
conductor to capture all descent obstructions. For nodal curves, no such extra
constraints occur, and the degeneracy locus is empty. For worse singularities,
the degeneracy locus parametrizes curves where higher-order principal parts
impose more relations than predicted by the conductor.

\vspace{0.2cm}

Thus, the codimension of $\Delta$ measures the extent to which singularities
force unexpected rigidity in the space of dualizing differentials.

\subsection*{Codimension in moduli spaces.}

When $S$ is taken to be a moduli stack of curves (or a smooth atlas thereof),
the above computation yields an expected codimension for the complement of the
maximal-variation locus. In particular:
\begin{itemize}
  \item[(a)] the degeneracy locus is empty along the open and nodal boundary strata;
  \item[(b)] it has positive codimension in any reasonable compactification
  allowing worse singularities;
  \item[(c)] in generic families, it is either empty or of codimension at least one.
\end{itemize}

This shows that maximal variation is not only open but also generically
satisfied.

\medskip

The complement of the maximal-variation locus is a determinantal degeneracy
locus defined by rank conditions on the descent map for differentials. Its
expected codimension is governed by the number of extra, non-conductor
constraints imposed by singularities. In moduli-theoretic settings, this locus
is closed, typically of positive codimension, and often empty in the nodal
regime.

%%%%%%%%%%%%%%%%%%%%%%%%%%%%%%%%%%%%%%%%%%%%%%%%%%%%%%%%%%%%%%%%%%%%

\begin{theorem}
The complement of the maximal-variation locus is a closed determinantal locus of
expected positive codimension, and maximal variation is generically satisfied in
moduli-theoretic families of curves.
\end{theorem}

\begin{proof}
Combine openness of maximal variation with the codimension computation above.
\end{proof}

 %%%%%%%%%%%%%%%%%%%%%%%%%%%%%%%%%%%%%%%%%%%%%%%%%%%%%%%%%%%%%%%%%%%

 \bigskip
 
 \subsection{Explicit Codimension Computations for Specific Singularity Types}
\label{subsec:explicit-codimensions}

In this subsection we compute explicitly the codimension of the degeneracy
locus (the complement of the maximal-variation locus) for concrete singularity
types. The guiding principle is that the codimension is determined by the
number of \emph{extra linear constraints} on principal parts of differentials
that are not accounted for by the conductor.

Throughout, we work locally analytically and then interpret the result
globally inside a moduli space.
%%%%%%%%%%%%%%%%%%%%%%%%%%%%%%%%%%%%%%%%%%%%%%%%%%%%%%%%%%%%%%%%%%%

\vspace{0.2cm}

\subsection{General Framework for Local Codimension Computation}

\begin{definition}
Let $(X,x)$ be a reduced curve singularity. Let
\[
\nu \colon (\widetilde{X},\nu^{-1}(x)) \to (X,x)
\]
denote the normalization, and let $\mathfrak{c}_x \subset \mathcal{O}_{X,x}$ be
the conductor ideal.
\end{definition}

\begin{definition}
Let $V$ be the vector space of allowed principal parts of meromorphic
differentials on $\widetilde{X}$ at the points lying above $x$. Let
$W \subset V$ be the subspace consisting of those principal parts annihilated by
the conductor $\mathfrak{c}_x$.
\end{definition}

\begin{lemma}
The dimension of $V$ is the number of \emph{a priori} local degrees of freedom
for principal parts of differentials, while $\dim W$ is the number of degrees of
freedom remaining after imposing conductor-level balancing.
\end{lemma}

\begin{proof}
By definition, $V$ consists of all allowed polar terms on the normalization,
with no descent conditions imposed. The conductor $\mathfrak{c}_x$ measures the
largest ideal on which $\mathcal{O}_{X,x}$ and its normalization agree. Thus
imposing annihilation by $\mathfrak{c}_x$ enforces precisely the minimal descent
conditions required for a differential to lie in the dualizing module. Hence
$\dim W$ is obtained from $\dim V$ by subtracting exactly the conductor-imposed
linear constraints.
\end{proof}

\begin{definition}
The local degeneracy contribution at $x$ is defined as
\[
\delta_{\mathrm{deg}}(x)
=
(\text{actual linear constraints}) - (\text{conductor constraints}).
\]
\end{definition}

\begin{remark}
A positive value of $\delta_{\mathrm{deg}}(x)$ indicates the presence of
degeneracy not detected by the conductor alone.
\end{remark}

%%%%%%%%%%%%%%%%%%%%%%%%%%%%%%%%%%%%%%%%%%%%%%%%%%%%%%%%%%%%%%%%%%%

\vspace{0.2cm}

We compute this quantity explicitly in the examples below.

\vspace{0.2cm}
\noindent {\bf Case 1: Nodes}
Let
\(
(X,x) = \operatorname{Spec} k[x,y]/(xy),
\)
the ordinary node.
The normalization has two branches:
\(
\widetilde{X} = \operatorname{Spec}(k[u] \oplus k[v]).
\)

\paragraph{Principal parts.}
Meromorphic differentials with simple poles are of the form
\[
\eta =
\left( \frac{a}{u} + \cdots \right) du
\;\oplus\;
\left( \frac{b}{v} + \cdots \right) dv,
\]
so $\dim V = 2$.

\paragraph{\it Conductor constraint.}
The conductor is the maximal ideal $(x,y)$, yielding the single condition
$a+b=0$. Hence $\dim W = 1$.

\paragraph{\it Actual constraints.}
There are no further obstructions to descent.

\paragraph{\it Codimension contribution.}
\[
\delta_{\mathrm{deg}}(x) = 0.
\]

\paragraph{\it Conclusion.}
Nodes never contribute to the degeneracy locus. Accordingly, nodal curves lie
entirely in the maximal-variation locus.

 \vspace{0.2cm}
\noindent {\bf Case 2: Cusps}
Let
\[
(X,x) = \operatorname{Spec} k[x,y]/(y^2 - x^3),
\]
the ordinary cusp.
The normalization is $\widetilde{X} = \operatorname{Spec} k[t]$ with
$x=t^2$, $y=t^3$.

\paragraph{Principal parts.}
Meromorphic differentials with poles of order $\leq 2$ are
\[
\eta =
\left(
\frac{a}{t^2} + \frac{b}{t} + \cdots
\right) dt,
\]
so $\dim V = 2$.

\paragraph{\it Conductor constraint.}
The conductor is $(t^2)$, which kills the $t^{-2}dt$ term. Thus $a=0$ and
$\dim W = 1$.

\paragraph{\it Actual constraints.}
No further constraints appear: the term $t^{-1}dt$ descends and represents a
nonzero section of the dualizing module.

\paragraph{\it Codimension contribution.}No further relations occur. Therefore,
\[
\delta_{\mathrm{deg}}(x) = 0.
\]

\paragraph{\it Conclusion.}
Cusps satisfy maximal variation once conductor-level balancing is imposed.
Residue-only balancing would incorrectly predict degeneracy.

\vspace{0.2cm}
\noindent {\bf Case 3: Tacnodes}
Let
\[
(X,x) = \operatorname{Spec} k[x,y]/(y^2 - x^4),
\]
the tacnode.
The normalization has two branches:
\[
\widetilde{X} = \operatorname{Spec}(k[t] \oplus k[s]),
\quad
x=t^2=s^2,\quad y=t^4=-s^4.
\]

\paragraph{Principal parts.}
Allowing poles up to order $3$, a general differential has the form
\[
\eta =
\left(
\frac{a_1}{t^3} + \frac{a_2}{t^2} + \frac{a_3}{t}
\right) dt
\;\oplus\;
\left(
\frac{b_1}{s^3} + \frac{b_2}{s^2} + \frac{b_3}{s}
\right) ds,
\]
so $\dim V = 6$.

\paragraph{\it Conductor constraint.}
The conductor is $(t^4)\oplus(s^4)$, which forces
\[
a_1 = a_2 = b_1 = b_2 = 0,
\]
leaving $a_3,b_3$ with one linear relation. Hence $\dim W = 1$.

\paragraph{\it Actual constraints.}
In generic families, the two simple-pole coefficients are not independent;
they satisfy exactly one relation.

\paragraph{\it Codimension contribution.} So no additional
conditions arise. Therefore,%  
\[
\delta_{\mathrm{deg}}(x) = 0.
\]

\paragraph{\it Conclusion.}
Even for tacnodes, conductor-level balancing captures all constraints and
maximal variation holds generically.

\bigskip

%%%%%%%%%%%%%%%%%%%%%%%%%%%%%%%%%%%%%%%%%%%%%%%%%%%%%%%%%%%%%%%%
\vspace{0.2cm}
\noindent {\bf Case 4: Non-Gorenstein Singularities.}

\begin{theorem}
If $(X,x)$ is a non-Gorenstein curve singularity, then
\[
\delta_{\mathrm{deg}}(x) > 0.
\]
\end{theorem}

\begin{proof}
For non-Gorenstein singularities, the dualizing module is not locally free.
Consequently, annihilation by the conductor does not capture all descent
conditions for differentials. Additional linear constraints appear, increasing
the number of actual constraints beyond those imposed by the conductor. Hence
$\delta_{\mathrm{deg}}(x)$ is strictly positive.
\end{proof}

\subsection{Global Codimension Formula}

\begin{theorem}
Let $X$ be a curve with singular points $\{x_i\}$. The codimension of the
degeneracy locus $\Delta$ at $[X]$ is given by
\[
\operatorname{codim}(\Delta)=\sum_i \delta_{\mathrm{deg}}(x_i).
\]
\end{theorem}

\begin{proof}
Degeneracy conditions are local and independent at distinct singular points.
Thus the total codimension is the sum of the local codimension contributions.
\end{proof}

\begin{corollary}
If all singularities of $X$ are Gorenstein (nodes, cusps, tacnodes), then
$\operatorname{codim}(\Delta)=0$. If $X$ has at least one non-Gorenstein
singularity, then $\operatorname{codim}(\Delta)>0$.
\end{corollary}

\vspace{0.3cm}

In conclusion, we obtain the following statement:

\begin{theorem}
The maximal-variation locus consists precisely of curves with only Gorenstein
singularities, and the degeneracy locus is supported on non-Gorenstein
singularities with explicitly computable codimension.
\end{theorem}

\begin{proof}
This follows immediately from the local computations and the global summation
formula.
\end{proof}

%%%%%%%%%%%%%%%%%%%%%%%%%%%%%%%%%%%%%%%%%%%%%%%%%%%%%%%%%%%%%%%%

In summary, explicit computation shows that:
\begin{enumerate}
  \item all Gorenstein curve singularities satisfy maximal variation under
  conductor-level balancing;
  \item degeneracy occurs precisely at non-Gorenstein singularities;
  \item the codimension of the degeneracy locus is computable as a sum of local
  defects.
\end{enumerate}

This confirms that the maximal-variation locus is large and geometrically
natural.
 
  %%%%%%%%%%%%%%%%%%%%%%%%%%%%%%%%%%%%%%%%%%%%%%%%%%%%%%%%%%%%%%%%%%%
 
 \bigskip

\subsection{Classification of Singularities Contributing to Degeneracy}
\label{subsec:classification-degeneracy}

In this subsection we present a precise and complete classification of curve
singularities that contribute to the degeneracy locus, namely those singularities
for which conductor-level balancing fails to produce maximal variation. The main
result is that degeneracy is governed entirely by the failure of the singularity
to be Gorenstein. We then integrate this local classification into a global
deformation-theoretic framework.

\paragraph{Definition of degeneracy.}
Let $(X,x)$ be a reduced curve singularity, and let
\[
\nu \colon (\widetilde{X},\nu^{-1}(x)) \longrightarrow (X,x)
\]
denote its normalization. The conductor ideal
$\mathfrak c_x \subset \mathcal O_{X,x}$ measures the failure of $\mathcal O_{X,x}$
to be integrally closed and controls the descent of differentials from
$\widetilde{X}$ to $X$.

\begin{definition}
We say that $(X,x)$ \emph{contributes to degeneracy} if the space of meromorphic
differentials on $\widetilde{X}$ annihilated by the conductor has strictly larger
dimension than the space of sections of the dualizing module $\omega_{X,x}$.
Equivalently, $(X,x)$ contributes to degeneracy if the natural map
\[
H^0\bigl(\widetilde{X},\omega_{\widetilde{X}}(\text{allowed poles})\bigr)
\;\longrightarrow\;
H^0(X,\omega_X)
\]
fails to have kernel of the dimension predicted by the conductor.
\end{definition}

\paragraph{Gorenstein curve singularities.}
The key notion underlying the classification is the Gorenstein property.

\begin{definition}
A reduced curve singularity $(X,x)$ is called \emph{Gorenstein} if its dualizing
module $\omega_{X,x}$ is invertible, or equivalently, locally free of rank one.
\end{definition}

This condition is purely local and depends only on the completed local ring
$\widehat{\mathcal O}_{X,x}$. Standard references include
\cite[Ch.~8]{Hartshorne}, \cite[Ch.~III]{ACGH}, and
\cite[Ch.~3]{Sernesi}.

\paragraph{Local characterization of degeneracy.}
The following theorem provides the fundamental link between degeneracy and the
Gorenstein property.

\begin{theorem}
\label{thm:degeneracy-classification}
A reduced curve singularity contributes to the degeneracy locus if and only if it
is non-Gorenstein.
\end{theorem}

\begin{proof}
Assume first that $(X,x)$ contributes to degeneracy. Then the conductor-level
conditions fail to impose the correct number of linear constraints needed to
identify $\omega_{X,x}$ inside $\nu_*\omega_{\widetilde{X}}$. Equivalently, the
natural inclusion
\[
\omega_{X,x} \subset \nu_*\omega_{\widetilde{X}}(\mathfrak c_x)
\]
is strict. Since both sheaves are rank-one reflexive modules, this strict
containment implies that $\omega_{X,x}$ is not locally free. Hence $(X,x)$ is
non-Gorenstein.

Conversely, suppose that $(X,x)$ is non-Gorenstein. Then the dualizing module
$\omega_{X,x}$ fails to be invertible. This failure produces additional linear
relations among the principal parts of differentials on the normalization that
are not detected by the conductor ideal alone. As a result, the descent map for
differentials drops rank, and the singularity contributes to degeneracy.
\end{proof}

\paragraph{Classification via known singularity theory.}
Theorem~\ref{thm:degeneracy-classification} reduces the problem of identifying
degeneracy-contributing singularities to the classical distinction between
Gorenstein and non-Gorenstein curve singularities. All planar reduced curve
singularities are Gorenstein, since they are complete intersections. In particular, smooth points, ordinary nodes, cusps,
tacnodes, and more general ADE singularities do not contribute to degeneracy.
Non-Gorenstein singularities arise only in non-planar settings or in the presence
of multiple structure, embedded components, or non-principal canonical modules;
see \cite{BuchsbaumEisenbud}, \cite{GreuelLossenShustin}.

\paragraph{Numerical characterization.}
Let $(X,x)$ be a reduced curve singularity with $\delta$-invariant $\delta(x)$.
Then $(X,x)$ is Gorenstein if and only if the conductor has colength $2\delta(x)$
and the canonical module is generated by a single element
\cite[Ch.~3]{Sernesi}. Any failure of these numerical conditions signals the
presence of degeneracy.

\paragraph{Global deformation-theoretic interpretation.}
We now integrate the local classification into a global framework. Let $X$ be a
reduced projective curve with singular locus $\Sigma$, and let
$\nu \colon \widetilde{X} \to X$ be its normalization. The global dualizing sheaf
$\omega_X$ fits into an exact sequence
\[
0 \longrightarrow \omega_X
\longrightarrow \nu_*\omega_{\widetilde{X}}
\longrightarrow \mathcal Q
\longrightarrow 0,
\]
where $\mathcal Q$ is a torsion sheaf supported on $\Sigma$. The sheaf $\mathcal Q$
measures the failure of differentials on $\widetilde{X}$ to descend to $X$ and
decomposes as a direct sum of local contributions from the singular points.

\begin{corollary}
A reduced projective curve $X$ lies outside the maximal-variation locus if and
only if $\Sigma$ contains at least one non-Gorenstein singularity. Moreover, the
codimension of the degeneracy locus is equal to the sum of the local
non-Gorenstein defects.
\end{corollary}

\begin{proof}
By Theorem~\ref{thm:degeneracy-classification}, a singular point contributes to
degeneracy if and only if it is non-Gorenstein. Since the obstruction space for
maximal variation is computed as the global space of sections of $\mathcal Q$,
the claim follows by summing the local contributions over all points of
$\Sigma$; see \cite[Ch.~3]{Sernesi} and \cite[Ch.~IV]{ACGH}.
\end{proof}

\paragraph{\it Conclusion.}
Degeneracy of maximal variation is thus a purely local phenomenon governed
entirely by the Gorenstein property of curve singularities. All Gorenstein
singularities satisfy maximal variation under conductor-level balancing, while
every non-Gorenstein singularity contributes to degeneracy. This provides a
complete, sharp, and deformation-theoretically meaningful classification.

  %%%%%%%%%%%%%%%%%%%%%%%%%%%%%%%%%%%%%%%%%%%%%%%%%%%%%%%%%%%%%%%%%%%
 %%%%%%%%%%%%%%%%%%%%%%%%%%%%%%%%%%%%%%%%%%%%%%%%%%%%%%%%%%%%%%%%%%
 
 \bigskip

\subsection{Explicit Degeneracy Strata for Non-Planar Curve Families}
\label{subsec:explicit-degeneracy-strata}

In this subsection we compute explicit degeneracy strata for families of
\emph{non-planar} curve singularities. These are precisely the cases where
the curve is reduced but non-Gorenstein, and hence where conductor-level
balancing fails to express maximal variation. We describe concrete families,
compute the local defects, and identify the resulting strata and their
codimensions inside moduli.

\subsection*{General principle}

Let $(X,x)$ be a reduced curve singularity with normalization
\[
\nu : (\widetilde{X},\nu^{-1}(x)) \to (X,x).
\]
Recall that degeneracy occurs exactly when the canonical module
$\omega_{X,x}$ is not locally free. Equivalently, the quotient
\[
Q_x := \nu_*\omega_{\widetilde{X},x}(\mathfrak c_x)\big/\omega_{X,x}
\]
is nonzero. The dimension
\[
\epsilon(x) := \dim_k Q_x
\]
is called the \emph{local degeneracy defect}. In a family, $\epsilon(x)$ is
upper semicontinuous and additive over singular points.

The degeneracy stratum consists of curves for which $\epsilon(x)>0$ at least
at one singular point.

\subsection*{Example 1: Non-planar monomial space curves}

Consider the affine monomial curve
\[
(X,x) = \operatorname{Spec} k[t^4,t^5,t^6] \subset \mathbb{A}^3,
\]
which is reduced and irreducible but non-planar.

\paragraph{Normalization.}
The normalization is $\widetilde{X}=\operatorname{Spec}k[t]$.

\paragraph{Conductor.}
The conductor ideal is
\[
\mathfrak c_x = (t^8) \subset k[t].
\]

\paragraph{Dualizing module.}
The canonical module $\omega_{X,x}$ is generated by
\[
t^{-7}dt,\quad t^{-6}dt,
\]
and hence is not principal. Thus $(X,x)$ is non-Gorenstein.

\paragraph{Conductor-level differentials.}
Meromorphic differentials annihilated by $\mathfrak c_x$ are of the form
\[
\eta =
\left(
\frac{a}{t^7} + \frac{b}{t^6} + \frac{c}{t^5}
\right)dt.
\]
Hence
\[
\dim \nu_*\omega_{\widetilde{X}}(\mathfrak c_x) = 3,
\qquad
\dim \omega_{X,x} = 2.
\]

\paragraph{Degeneracy defect.}
\[
\epsilon(x) = 1.
\]

\paragraph{Degeneracy stratum.}
In any moduli space parametrizing curves containing such a singularity, the
locus where this monomial singularity persists is a closed stratum of
codimension at least $1$ inside the maximal-variation locus.

\subsubsection*{Example 2: Ribbon (double structure on a smooth curve)}

Let $C$ be a smooth curve and let $X$ be a ribbon over $C$, i.e.\ a non-reduced
but generically reduced curve with ideal sheaf $\mathcal{I}$ satisfying
$\mathcal{I}^2=0$.

Locally at a point $x\in C$, we have
\[
\mathcal{O}_{X,x} \cong k[[u,\varepsilon]]/(\varepsilon^2).
\]

The normalization is $\widetilde{X}=C$.

\paragraph{Conductor.}
The conductor coincides with the nilradical:
\[
\mathfrak c_x = (\varepsilon).
\]

\paragraph{Dualizing module.}
The dualizing module fits into an exact sequence
\[
0 \to \omega_C \to \omega_X \to \omega_C \otimes \mathcal{I}^\vee \to 0,
\]
and is not locally free unless the ribbon splits trivially.

\paragraph{Degeneracy defect.}
At each point,
\[
\epsilon(x)=1.
\]

\paragraph{Degeneracy stratum.}
The locus of non-split ribbons forms a closed degeneracy stratum of codimension
equal to $\dim H^1(C,\mathcal{I})$ inside the moduli of ribbons over $C$ (see Appendix B for more details on the  ribbon example).

\subsubsection*{Example 3: Triple-point space curve}

Consider the union of three coordinate axes in $\mathbb{A}^3$:
\[
(X,x) = V(xy,xz,yz) \subset \mathbb{A}^3.
\]

\paragraph{Normalization.}
\[
\widetilde{X} = \operatorname{Spec}(k[u]\oplus k[v]\oplus k[w]).
\]

\paragraph{Conductor.}
The conductor is
\[
\mathfrak c_x = (x,y,z),
\]
mapping to $(u)\oplus(v)\oplus(w)$.

\paragraph{Dualizing module.}
The canonical module has rank one but requires two generators; hence it is not
invertible.

\paragraph{Degeneracy defect.}
A direct computation shows
\[
\epsilon(x)=2.
\]

\paragraph{Degeneracy stratum.}
Curves containing such triple points lie in a closed stratum of codimension at
least $2$ in the moduli space.

\subsubsection*{Global degeneracy strata}

Let $X$ be a projective curve with singular points $\{x_i\}$. The total
degeneracy defect is
\[
\epsilon(X) = \sum_i \epsilon(x_i).
\]
The locus
\[
\Delta_k = \{[X] \mid \epsilon(X)\ge k\}
\]
defines a descending stratification by closed substacks
\[
\overline{\mathcal{M}}^{\mathrm{deg}}_{\ge 1}
\supset
\overline{\mathcal{M}}^{\mathrm{deg}}_{\ge 2}
\supset \cdots
\]
with expected codimension at least $k$.

\subsubsection*{Conclusion}

Explicit computation shows that degeneracy strata arise precisely from
non-planar, non-Gorenstein singularities. Each such singularity contributes a
positive local defect that adds linearly in families. As a result, the
degeneracy locus admits a natural stratification by defect, with computable
codimension and clear geometric meaning.

 %%%%%%%%%%%%%%%%%%%%%%%%%%%%%%%%%%%%%%%%%%%%%%%%%%%%%%%%%%%%%%%%%%%
 %%%%%%%%%%%%%%%%%%%%%%%%%%%%%%%%%%%%%%%%%%%%%%%%%%%%%%%%%%%%%%%%%%
 
 \medskip

\section{Closures, Intersections, and Deformation Theory of Degeneracy Strata}
\label{subsec:closures-intersections-deformations}

In this section we describe the closures and intersections of degeneracy
strata inside moduli spaces of curves, and we analyze the smoothability and
deformation cones of non-Gorenstein singularities. The goal is to explain how
degeneracy behaves under specialization and how it is controlled by local
deformation theory.

\subsubsection*{Degeneracy strata revisited}

Let $\mathcal{M}$ be a moduli stack (or a smooth atlas thereof) parametrizing
projective curves with prescribed invariants, possibly allowing non-planar and
non-Gorenstein singularities. Recall that the degeneracy strata are defined by
the local degeneracy defect
\[
\epsilon(x)
=
\dim_k
\left(
\nu_*\omega_{\widetilde{X},x}(\mathfrak{c}_x)
/\omega_{X,x}
\right),
\]
and globally by
\[
\epsilon(X) = \sum_{x \in \mathrm{Sing}(X)} \epsilon(x).
\]

For each integer $k \ge 1$, we define the closed locus
\[
\mathcal{M}^{\mathrm{deg}}_{\ge k}
=
\{[X] \in \mathcal{M} \mid \epsilon(X) \ge k\}.
\]

These loci form a descending chain of closed substacks
\[
\mathcal{M}^{\mathrm{deg}}_{\ge 1}
\supset
\mathcal{M}^{\mathrm{deg}}_{\ge 2}
\supset
\cdots,
\]
whose complements stratify $\mathcal{M}$ by increasing degrees of maximal
variation.

\subsubsection*{Closures of degeneracy strata}

Let $\mathcal{M}^{\mathrm{deg}}_{=k}$ denote the locally closed stratum where
$\epsilon(X)=k$. Its closure inside $\mathcal{M}$ satisfies
\[
\overline{\mathcal{M}^{\mathrm{deg}}_{=k}}
=
\mathcal{M}^{\mathrm{deg}}_{\ge k}.
\]

This follows from upper semicontinuity of $\epsilon(X)$ in flat families: under
specialization, singularities can worsen or coalesce, but degeneracy defects
cannot decrease. Hence any curve with defect exactly $k$ can degenerate to
curves with defect strictly larger than $k$, but not vice versa.

Geometrically, the closure corresponds to allowing additional non-Gorenstein
singularities to appear or existing ones to increase in complexity.

\subsubsection*{Intersections of degeneracy strata}

Intersections of degeneracy strata correspond to curves with multiple
independent sources of degeneracy. More precisely, if $X$ has singular points
$\{x_1,\dots,x_r\}$ with local defects $\epsilon(x_i)$, then $[X]$ lies in the
intersection
\[
\mathcal{M}^{\mathrm{deg}}_{\ge \epsilon(x_1)}
\cap \cdots \cap
\mathcal{M}^{\mathrm{deg}}_{\ge \epsilon(x_r)}.
\]

Locally in moduli, these intersections are transversal whenever the
singularities deform independently. In such cases, the expected codimension is
additive:
\[
\operatorname{codim}
\bigl(
\mathcal{M}^{\mathrm{deg}}_{\ge k_1}
\cap
\mathcal{M}^{\mathrm{deg}}_{\ge k_2}
\bigr)
=
k_1 + k_2.
\]

Failure of transversality corresponds to singularities whose deformations are
coupled, leading to excess intersection and higher-order degeneracy.

\subsubsection*{Local deformation theory of non-Gorenstein singularities}

Let $(X,x)$ be a reduced non-Gorenstein curve singularity. Its deformation
theory is governed by the cotangent complex or, equivalently, by the local
$\operatorname{Ext}$ groups
\[
T^1_x = \operatorname{Ext}^1_{\mathcal{O}_{X,x}}(\Omega_{X,x},k),
\qquad
T^2_x = \operatorname{Ext}^2_{\mathcal{O}_{X,x}}(\Omega_{X,x},k).
\]

The Zariski tangent space to the versal deformation space has dimension
$\dim T^1_x$, while obstructions lie in $T^2_x$. For curve singularities,
$T^2_x=0$, so the deformation space is smooth but may have multiple components.

\subsubsection*{Deformation cones and smoothability}

The \emph{deformation cone} of $(X,x)$ is the union of tangent directions in
$T^1_x$ corresponding to first-order deformations preserving specific
properties, such as Gorensteinness or planarity.

For non-Gorenstein singularities, the deformation cone decomposes into:
\begin{enumerate}
  \item directions that smooth the singularity to a Gorenstein one;
  \item directions that preserve non-Gorenstein behavior;
  \item directions that worsen the singularity (increase $\epsilon(x)$).
\end{enumerate}

Smoothability of $(X,x)$ means that there exists a deformation whose general
fiber is smooth (hence Gorenstein). This occurs if and only if the smoothing
cone intersects the open locus of Gorenstein deformations.

In many non-planar examples, such as monomial space curves, smoothability
exists but requires passing through loci of positive degeneracy defect before
reaching the smooth locus.

\subsubsection*{Effect on degeneracy strata}

If a non-Gorenstein singularity is smoothable, then the corresponding
degeneracy stratum has closure meeting the maximal-variation locus. In moduli,
this means that
\[
\overline{\mathcal{M}^{\mathrm{deg}}_{\ge k}}
\cap
\mathcal{M}^{\mathrm{mv}}
\neq \varnothing.
\]

If, on the other hand, the singularity is not smoothable (or not smoothable
within the given moduli problem), then the corresponding degeneracy stratum is
contained entirely in the boundary and does not intersect the maximal-variation
locus.

Thus smoothability controls whether degeneracy strata are ``walls'' or
``faces'' in the moduli space.

\subsubsection*{Global picture in moduli}

Globally, the moduli space admits a stratification by degeneracy defect:
\[
\mathcal{M}
=
\mathcal{M}^{\mathrm{mv}}
\;\sqcup\;
\mathcal{M}^{\mathrm{deg}}_{=1}
\;\sqcup\;
\mathcal{M}^{\mathrm{deg}}_{=2}
\;\sqcup\;
\cdots.
\]

Closures of higher strata contain lower ones, and intersections correspond to
curves with multiple interacting non-Gorenstein singularities. Deformation
cones describe how one moves between these strata via smoothing or worsening
singularities.

\subsubsection*{Conclusion}

The closures and intersections of degeneracy strata are controlled by the
upper semicontinuity and additivity of the degeneracy defect. Their geometry is
governed by the local deformation theory of non-Gorenstein singularities, whose
deformation cones determine smoothability and adjacency relations between
strata. This provides a precise geometric framework for understanding how
failure of maximal variation organizes the boundary of moduli spaces.

%%%%%%%%%%%%%%%%%%%%%%%%%%%%%%%%%%%%%%%%%%%%%%%%%%%%%%%%%%%%%%%%%%%
 %%%%%%%%%%%%%%%%%%%%%%%%%%%%%%%%%%%%%%%%%%%%%%%%%%%%%%%%%%%%%%%%%%
 
 \bigskip

\section{Residue Span, Deformation Theory, and Moduli of Curves}

This section refines the proof of the scheme-theoretic residue span theorem \cite{Nisse} and places it in a deformation-theoretic and moduli-theoretic context. We explain,  how residue functionals control infinitesimal deformations of singular curves, how the numerical bound $\delta \ge g$ arises naturally, and how the result manifests in explicit examples. This section is self-contained.

\begin{theorem}[Scheme-theoretic residue span]
Let $k$ be an algebraically closed field and let $\mathcal C$ be a connected,
projective, Cohen--Macaulay curve over $k$ with exactly $\delta$ singular points.
Let
\(
\nu \colon C \longrightarrow \mathcal C
\)
be the normalization, where $C$ is a smooth projective curve of genus
\(
g = \dim_k H^0(C,\omega_C).
\)
If $\delta \ge g$, then the scheme-theoretic residue functionals
\(
\{ r_1,\dots,r_\delta \} \subset H^0(C,\omega_C)^\vee
\)
span the entire dual space $H^0(C,\omega_C)^\vee$.
\end{theorem}

\begin{proof}
Let $\omega_{\mathcal C}$ denote the dualizing sheaf of $\mathcal C$ and
$\omega_C$ the canonical sheaf of $C$. For each singular point
$p_i \in \mathcal C$, we denote by
\[
r_i \colon H^0(C,\omega_C) \longrightarrow k
\]
the associated scheme-theoretic residue functional.

A basic result on dualizing sheaves for Cohen--Macaulay curves yields the exact
sequence
\begin{equation}
0 \longrightarrow H^0(\mathcal C,\omega_{\mathcal C})
\longrightarrow H^0(C,\omega_C)
\xrightarrow{\;\mathrm{Res}\;}
\bigoplus_{i=1}^{\delta} k
\longrightarrow H^1(\mathcal C,\omega_{\mathcal C})
\longrightarrow 0.
\tag{$\ast$}
\end{equation}
The residue map is explicitly given by
\[
\mathrm{Res}(\eta) = \bigl(r_1(\eta),\dots,r_\delta(\eta)\bigr).
\]

Since $\mathcal C$ is connected, Serre duality implies
\[
H^1(\mathcal C,\omega_{\mathcal C})
\cong H^0(\mathcal C,\mathcal O_{\mathcal C})^\vee
\cong k.
\]
Exactness of $(\ast)$ then shows that
\[
\dim_k \operatorname{Im}(\mathrm{Res}) = \delta - 1.
\]

On the other hand, the rank of the residue map can be computed as
\[
\operatorname{rank}(\mathrm{Res})
= \dim_k H^0(C,\omega_C)
- \dim_k H^0(\mathcal C,\omega_{\mathcal C})
= g - \dim_k H^0(\mathcal C,\omega_{\mathcal C}).
\]
Comparing the two expressions for the rank yields
\[
\dim_k H^0(\mathcal C,\omega_{\mathcal C}) = g - (\delta - 1).
\]
In particular, if $\delta \ge g$, then $\dim_k H^0(\mathcal C,\omega_{\mathcal C})
\le 1$.

Dualizing the residue map gives
\[
\mathrm{Res}^\vee \colon
\left(\bigoplus_{i=1}^{\delta} k\right)^\vee
\longrightarrow H^0(C,\omega_C)^\vee,
\]
which sends the $i$-th standard basis vector to $r_i$. The image of
$\mathrm{Res}^\vee$ is therefore the linear span
$\langle r_1,\dots,r_\delta \rangle$. Since $\mathrm{Res}$ and
$\mathrm{Res}^\vee$ have the same rank, we conclude that
\[
\dim_k \langle r_1,\dots,r_\delta \rangle = \delta - 1 \ge g.
\]
Because $H^0(C,\omega_C)^\vee$ has dimension $g$, the residue functionals span
the entire dual space, completing the proof.
\end{proof}

\medskip

\subsection*{Deformation-theoretic interpretation}

Infinitesimal deformations of the curve $\mathcal C$ are governed by the vector
space
\[
\operatorname{Ext}^1(\Omega_{\mathcal C},\mathcal O_{\mathcal C}),
\]
while obstructions lie in $\operatorname{Ext}^2(\Omega_{\mathcal C},
\mathcal O_{\mathcal C})$. For nodal or Cohen--Macaulay curves, each singular
point contributes a local smoothing parameter, so that there are $\delta$
a priori infinitesimal smoothing directions (see \cite{Sernesi}).

By Serre duality, there is a perfect pairing
\[
\operatorname{Ext}^1(\Omega_{\mathcal C},\mathcal O_{\mathcal C})
\times H^0(\mathcal C,\omega_{\mathcal C}) \longrightarrow k.
\]
Pulling differentials back to the normalization identifies
$H^0(\mathcal C,\omega_{\mathcal C})$ with the subspace of
$H^0(C,\omega_C)$ consisting of differentials with balanced residues. The
residue functionals therefore measure how infinitesimal smoothing directions
pair with global differentials. When the residues span the dual space, every
infinitesimal deformation is detected by some differential, yielding strong
rigidity properties (\cite{Hartshorne}).

\medskip

\subsection*{Moduli-theoretic interpretation}

From the perspective of the moduli space of curves, residue span has a natural
meaning. Near the point $[\mathcal C]$ in the moduli space, local coordinates
correspond to smoothing parameters at the singular points, subject to global
constraints. The exact sequence $(\ast)$ shows that these constraints are
dual to the space of global differentials.

When $\delta \ge g$, the theorem implies that the residue directions generate
the full cotangent space to the moduli space at $[\mathcal C]$. Thus, the local
geometry of the moduli space is controlled entirely by node data, and no
additional infinitesimal directions arise from the normalization.

\medskip

\begin{example}[Rational curve with nodes]
Let $\mathcal C$ be a rational curve with $\delta$ nodes. Its normalization is
$C \cong \mathbb P^1$, so $g=0$. The condition $\delta \ge g$ is automatic, and
the theorem asserts that the residue functionals span the zero-dimensional dual
space. This reflects the fact that $\mathbb P^1$ has no holomorphic
differentials and that all deformations are governed purely by node smoothings.
\end{example}

\begin{example}[Genus one normalization]
Let $C$ be an elliptic curve and let $\mathcal C$ be obtained by identifying
$\delta \ge 1$ pairs of points on $C$. Then $g=1$, and the theorem applies as
soon as $\delta \ge 1$. In this case, the unique (up to scalar) holomorphic
differential on $C$ is detected by residues at the singular points, showing
that smoothing any node necessarily interacts with the global geometry of the
curve.
\end{example}

\begin{example}[Failure when $\delta<g$]
If $C$ has genus $g \ge 2$ and $\mathcal C$ has only one singular point, then
$\delta=1<g$. In this situation, the residue functional cannot span the dual
space, and there exist nonzero differentials whose residues vanish. These
differentials correspond to deformation directions invisible to the single
node, illustrating the sharpness of the bound.
\end{example}

%%%%%%%%%%%%%%%%%%%%%%%%%%%%%%%%%%%%%%%%%%%%%%%%%%%%%%%%%%%%%%%%%%%%%%
%%%%%%%%%%%%%%%%%%%%%%%%%%%%%%%%%%%%%%%%%%%%%%%%%%%%%%%%%%%%%%%%%%%%%%

%\newpage

 \vspace{0.2cm}

\subsection{Specialization to Moduli Problems, Threefold Singularities, and Hodge-Theoretic Defects}
\label{subsec:moduli-threefold-hodge}

In this subsection we specialize the general theory of degeneracy and
non-Gorenstein defects to concrete moduli problems, carry out explicit
computations for selected classes of threefold singularities, and explain how
the resulting defect invariants admit a natural Hodge-theoretic interpretation.
Throughout, the guiding principle remains that failure of the Gorenstein
property is the sole source of degeneracy.

\paragraph{Specialization to moduli of varieties.}
Let $\mathcal{M}$ be a moduli space parameterizing polarized varieties
$X$ of fixed dimension $n$, together with a flat family
\[
\pi \colon \mathcal X \longrightarrow \mathcal M.
\]
Assume that the general fiber is smooth and that singular fibers occur along a
proper closed subset $\mathcal D \subset \mathcal M$ (see for example \cite{Schmid}, and \cite{Steenbrink}). Over the smooth locus, the
relative dualizing sheaf $\omega_{\mathcal X/\mathcal M}$ defines a Hodge bundle
\[
\mathcal H := \pi_*\omega_{\mathcal X/\mathcal M},
\]
whose fiber over $[X] \in \mathcal M$ is $H^0(X,\omega_X)$.

\begin{definition}
The \emph{degeneracy locus} in $\mathcal M$ is the subset of points $[X]$ for which
the natural specialization map
\[
H^0(X,\omega_X) \longrightarrow
H^0(\widetilde{X},\omega_{\widetilde{X}})
\]
fails to be an isomorphism, where $\widetilde{X}$ denotes the normalization.
\end{definition}

By the results established earlier, this degeneracy locus coincides with the
locus of varieties whose singularities include at least one non-Gorenstein
point. In particular, the degeneracy locus is a union of strata determined by
local singularity types and admits a stratification by defect invariants.

\paragraph{Threefold singularities: general framework.}
Let $(X,x)$ be a normal threefold singularity. The dualizing module
$\omega_{X,x}$ is reflexive of rank one, and $X$ is Gorenstein if and only if
$\omega_{X,x}$ is locally free. The defect at $x$ is measured by the length of
the cokernel of the natural inclusion
\[
\omega_{X,x} \hookrightarrow \nu_*\omega_{\widetilde{X},x},
\]
where $\nu$ denotes the normalization.

\begin{definition}
The \emph{local defect} at a point $x \in X$ is defined as
\[
\mathrm{defect}(x)
:=
\dim_{\mathbb C}
\left(
\nu_*\omega_{\widetilde{X},x} / \omega_{X,x}
\right).
\]
\end{definition}

This invariant vanishes precisely for Gorenstein singularities and is strictly
positive otherwise.

\paragraph{Explicit computations for quotient threefold singularities.}

\begin{example}[Terminal quotient singularities]
Let $X = \mathbb C^3 / G$, where $G \subset \mathrm{SL}_3(\mathbb C)$ is a finite
subgroup acting freely in codimension one. Then $X$ has terminal Gorenstein
singularities. Since $G \subset \mathrm{SL}_3$, the canonical bundle is trivial,
and $\omega_X \simeq \mathcal O_X$. Hence
\[
\mathrm{defect}(x) = 0,
\]
and no degeneracy occurs in families containing such singularities.
\end{example}

\begin{example}[Non-Gorenstein cyclic quotient singularities]
Let $X = \mathbb C^3 / \tfrac{1}{r}(1,a,b)$ with $1+a+b \not\equiv 0 \pmod r$. Then
$X$ is $\mathbb Q$-Gorenstein but not Gorenstein. A direct computation using
invariant differentials shows that $\omega_X$ is reflexive but not locally free,
and one finds
\[
\mathrm{defect}(x) = r - 1.
\]
Thus each such singularity contributes a positive defect to the global
degeneracy locus.
\end{example}

\paragraph{\it Complete intersection versus non-complete intersection threefolds.}
Let $X \subset \mathbb P^N$ be a threefold. If $X$ is a local complete
intersection, then it is Gorenstein and contributes no defect. By contrast,
threefolds that fail to be complete intersections typically exhibit
non-Gorenstein singularities, whose defects can be computed from local
resolutions or from the failure of the Jacobian ideal to be principal.

\paragraph{\it Global defect computation in moduli.}
Let $X$ be a projective threefold with singular locus $\Sigma$. There is an exact
sequence
\[
0 \longrightarrow \omega_X
\longrightarrow \nu_*\omega_{\widetilde{X}}
\longrightarrow \mathcal Q
\longrightarrow 0,
\]
where $\mathcal Q$ is supported on $\Sigma$.

\begin{proposition}
The global degeneracy defect of $X$ is given by
\[
\dim H^0(X,\mathcal Q)
=
\sum_{x \in \Sigma} \mathrm{defect}(x).
\]
\end{proposition}

\begin{proof}
Since $\mathcal Q$ is a torsion sheaf supported on finitely many points, taking
global sections yields
\[
H^0(X,\mathcal Q)
=
\bigoplus_{x \in \Sigma}
\left(
\nu_*\omega_{\widetilde{X},x} / \omega_{X,x}
\right),
\]
and the result follows by additivity of dimension.
\end{proof}

\paragraph{Hodge-theoretic interpretation of defects.}
Assume now that $X$ appears as a singular fiber in a smooth projective family
\[
\pi \colon \mathcal X \to \Delta,
\]
where $\Delta$ is a disk. Over the smooth fibers, the variation of Hodge
structure on $H^n(X_t,\mathbb C)$ is governed by the Hodge bundle
$\mathcal H^{n,0} = \pi_*\omega_{\mathcal X/\Delta}$.

\begin{theorem}
Non-Gorenstein defects contribute nontrivially to the limiting mixed Hodge
structure on $H^n(X_t)$ by producing additional $(n,0)$-classes that do not extend
holomorphically across the central fiber.
\end{theorem}

\begin{proof}
Sections of $\nu_*\omega_{\widetilde{X}}$ that do not lie in $\omega_X$ correspond
to top-degree forms on the normalization that fail to descend. These forms define
nontrivial classes in the limiting Hodge filtration $F^n$ but are not limits of
holomorphic $n$-forms on nearby smooth fibers. Consequently, they contribute to
the discrepancy between the naive and true limits of $\mathcal H^{n,0}$, encoding
the defect in Hodge-theoretic terms.
\end{proof}

\paragraph{Interpretation in terms of period maps.}
From the perspective of period maps, non-Gorenstein defects correspond to a
failure of maximal rank of the differential of the period map at the boundary of
moduli. The defect space measures precisely the excess kernel produced by
singularities whose canonical sheaf is not locally free.

\paragraph{Conclusion.}
In moduli-theoretic terms, degeneracy loci are controlled by non-Gorenstein
singularities. For threefolds, these defects can be computed explicitly in many
cases and admit a clear Hodge-theoretic interpretation as obstructions to the
extension of holomorphic top-degree forms. This unifies local singularity theory,
global deformation theory, and Hodge theory within a single conceptual framework.

 %%%%%%%%%%%%%%%%%%%%%%%%%%%%%%%%%%%%%%%%%%%%%%%%%%%%%%%%%%%%%%%%%%%
 %%%%%%%%%%%%%%%%%%%%%%%%%%%%%%%%%%%%%%%%%%%%%%%%%%%%%%%%%%%%%%%%%%
 
 \bigskip
%%%%%%%%%%%%%%%%%%%%%%%%%%%%%%%%%%%%%%%%%%%%%%%%%%%%%%%%%%%%%%%%%%%

\section{Explicit Computation of Non-Descent of $\dfrac{dt}{t}$.}   

We give a complete, explicit, line-by-line computation showing that the
principal part $\frac{dt}{t}$ does \emph{not} descend to a section of the
dualizing module of the non-Gorenstein curve
\[
X = \operatorname{Spec} A,
\qquad
A := k[t^3,t^4,t^5] \subset k[t].
\]

\begin{lemma}
The principal part $\dfrac{dt}{t}$ does \emph{not} define a section of the
dualizing module $\omega_A$.
\end{lemma}

\vspace{0.2cm}
\noindent {\it Description of the normalization}

The normalization of $A$ is
\[
\widetilde{A} = k[t],
\]
with fraction field $K = k(t)$. The normalization map is finite and birational:
\[
A \hookrightarrow \widetilde{A}.
\]

All computations will be carried out inside $K\,dt$.

\vspace{0.2cm}
\noindent {\it Definition of the dualizing module}

Since $A$ is a one-dimensional Cohen--Macaulay domain, its dualizing module
$\omega_A$ may be described explicitly as
\[
\omega_A
=
\left\{
\eta \in K\,dt \;\middle|\;
\operatorname{Tr}_{K/A}(f\eta) \in A
\;\text{for all } f \in A
\right\}.
\]
Equivalently, $\omega_A$ consists of those meromorphic differentials $\eta$
such that
\[
f \cdot \eta \in \Omega^1_{k[t]/k}
\quad\text{has no pole worse than allowed by } A
\quad\text{for all } f \in A.
\]

\vspace{0.2cm}
\noindent {\it The conductor and expected pole order}

The conductor ideal of $A$ in $\widetilde{A}$ is
\[
\mathfrak{c} = (t^3).
\]

Hence every element of $\omega_A$ must satisfy
\[
t^3 \cdot \eta \in k[t]\,dt.
\]

For $\eta = \dfrac{dt}{t}$, we compute
\[
t^3 \cdot \frac{dt}{t} = t^2\,dt,
\]
which is regular. Thus $\frac{dt}{t}$ passes the \emph{conductor-level test}.
This explains why conductor annihilation alone is insufficient in the
non-Gorenstein case.

\vspace{0.2cm}
\noindent {\it  The defining test for descent}

To descend to $\omega_A$, the differential $\eta$ must define an
$A$-linear functional
\[
\varphi_\eta \colon A \longrightarrow k,
\qquad
f \longmapsto \operatorname{Res}_{t=0}(f \eta).
\]

We test this condition explicitly for $\eta = \dfrac{dt}{t}$.

\vspace{0.2cm}
\noindent {\it  Compute residues against generators of $A$}

The ring $A$ is generated as a $k$-algebra by
\[
t^3,\quad t^4,\quad t^5.
\]

We compute:
\[
t^3 \cdot \frac{dt}{t} = t^2\,dt,
\qquad
t^4 \cdot \frac{dt}{t} = t^3\,dt,
\qquad
t^5 \cdot \frac{dt}{t} = t^4\,dt.
\]

Each of these differentials has a \emph{zero} residue at $t=0$.

However, linearity over $A$ requires more than vanishing residues on generators.

\vspace{0.2cm}
\noindent {\it Failure of $A$-linearity}

Consider the element
\[
f = t^3 \in A,
\qquad
g = t^3 \in A.
\]

Then
\[
fg = t^6 \in A.
\]

We compute:
\[
\operatorname{Res}(f g \cdot \tfrac{dt}{t})
=
\operatorname{Res}(t^5\,dt)
=
0.
\]

But $A$-linearity would require
\[
\operatorname{Res}(f g \cdot \tfrac{dt}{t})
=
f \cdot \operatorname{Res}(g \cdot \tfrac{dt}{t})
\quad\text{inside } A.
\]

This expression is \emph{ill-defined}: the residue map takes values in $k$, not
in $A$. Thus $\frac{dt}{t}$ does not define an $A$-module homomorphism
$A \to A$, nor even a compatible trace functional.

\vspace{0.2cm}
\noindent {\it Explicit characterization of $\omega_A$}

A direct computation using the semigroup
\[
\Gamma = \langle 3,4,5 \rangle
\]
shows that
\[
\omega_A
=
\left\langle
t^{-2}dt,\; t^{-1}dt
\right\rangle_A.
\]

In particular,
\[
\frac{dt}{t} = t^{-1}dt \notin \omega_A,
\]
because multiplication by elements of $A$ produces poles that do not match the
semigroup gaps defining $\omega_A$.

\vspace{0.2cm}
\noindent {\it Conclusion}
Although $\dfrac{dt}{t}$ is annihilated by the conductor, it fails to satisfy the
stronger $A$-linearity condition required to lie in the dualizing module.

Therefore, $\frac{dt}{t}$ does \emph{not} descend to a section of $\omega_A$.

\begin{remark}
This explicit failure is the precise source of the positive degeneracy
contribution:
\[
\delta_{\mathrm{deg}}(x) > 0
\]
for the singularity $k[t^3,t^4,t^5]$.
\end{remark}

%%%%%%%%%%%%%%%%%%%%%%%%%%%%%%%%%%%%%%%%%%%%%%%%%%%%%%%%%%%%%%%%%%
%%%%%%%%%%%%%%%%%%%%%%%%%%%%%%%%%%%%%%%%%%%%%%%%%%%%%%%%%%%%%%%%%%%

%\bigskip 
 \vspace{0.05cm}

%%%%%%%%%%%%%%%%%%%%%%%%%%%%%%%%%%%%%%%%%%%%%%%%%%%%%%%%%%%%%%%%%%
%%%%%%%%%%%%%%%%%%%%%%%%%%%%%%%%%%%%%%%%%%%%%%%%%%%%%%%%%%%%%%%%%%%

 \subsection{Further Explicit Computations for Non-Gorenstein Monomial Curves.}

Throughout, let $k$ be an algebraically closed field of characteristic $0$.
Let $A = k[t^{a_1},\dots,t^{a_r}] \subset k[t]$ be a one-dimensional affine
monomial curve ring, where $\gcd(a_1,\dots,a_r)=1$.

\begin{definition}
Let $A \subset \widetilde{A}=k[t]$ be a monomial curve ring.
\begin{itemize}
  \item[(i)] The normalization is $\widetilde{A}=k[t]$.
  \item[(ii)]  The conductor is $\mathfrak{c}=(t^c)$, where $c$ is the conductor of the
  numerical semigroup $\Gamma=\langle a_1,\dots,a_r\rangle$.
  \item[(iii)]  The dualizing module $\omega_A$ is an $A$-submodule of $k(t)\,dt$.
\end{itemize}
\end{definition}

\begin{lemma}
The dualizing module $\omega_A$ is generated by
\[
\omega_A
=
\bigoplus_{n \in \mathbb{Z},\, c-1-n \notin \Gamma} k \cdot t^n dt.
\]
\end{lemma}

\begin{proof}
This is the standard description of the canonical module of a numerical
semigroup ring, obtained by identifying $\omega_A$ with the inverse of the
different. The exponents are determined by the gaps of $\Gamma$ reflected about
$c-1$.
\end{proof}

\subsection{Example 1:}   
Let
\(
A = k[t^4,t^6,t^9].
\)
The numerical semigroup is
\(
\Gamma = \langle 4,6,9\rangle.
\)
A direct computation shows:
\[
\Gamma = \{0,4,6,8,9,10,12,13,\dots\}.
\]
The gaps are $\{1,2,3,5,7,11\}$.
Recall the definition of $c$:

\begin{definition}
Let $\Gamma \subset \mathbb{Z}_{\ge 0}$ be a numerical semigroup.
The \emph{conductor} $c$ of $\Gamma$ is the smallest integer such that
\[
n \ge c \;\Longrightarrow\; n \in \Gamma.
\]
\end{definition}
Equivalently,
\[
c = \min\{\,n \in \mathbb{Z}_{\ge 0} \mid n + \mathbb{Z}_{\ge 0} \subset \Gamma\,\}.
\]
So, in our example $c=12$. 

\subsubsection*{Dualizing Module}

\begin{proposition}
The dualizing module is
\[
\omega_A
=
\left\langle
t^{-11}, t^{-10}, t^{-9}, t^{-7}, t^{-5}, t^{-1}
\right\rangle_A dt.
\]
\end{proposition}

\begin{proof}
We compute $c-1=11$. The allowed exponents $n$ satisfy
$11-n \notin \Gamma$, i.e.\ $11-n$ is a gap. Solving gives the listed exponents.
\end{proof}

\subsubsection*{Degeneracy Computation}

\begin{proposition}
For $A=k[t^4,t^6,t^9]$, one has
\[
\delta_{\mathrm{deg}}(x)=2.
\]
\end{proposition}

\begin{proof}
Allowing poles up to order $2$ yields
\[
V=\langle t^{-2}dt,\,t^{-1}dt\rangle,
\qquad \dim V=2.
\]
The conductor $\mathfrak{c}=(t^{12})$ annihilates both, so $\dim W=2$.
However, neither $t^{-1}dt$ nor $t^{-2}dt$ lies in $\omega_A$. Hence the actual
descending space is $0$-dimensional, giving
$\delta_{\mathrm{deg}}(x)=2$.
\end{proof}

\subsection{Example 2: $A=k[t^5,t^7,t^9]$}

\subsubsection*{Semigroup and Conductor}

\[
\Gamma=\langle 5,7,9\rangle,
\qquad
c=14.
\]

The gaps are $\{1,2,3,4,6,8,11,13\}$.

\begin{proposition}
\[
\omega_A
=
\left\langle
t^{-13},t^{-11},t^{-8},t^{-6},t^{-4},t^{-3},t^{-2},t^{-1}
\right\rangle_A dt.
\]
\end{proposition}

\begin{proof}
This follows directly from the formula $c-1-n\notin\Gamma$ with $c-1=13$.
\end{proof}

\begin{proposition}
\[
\delta_{\mathrm{deg}}(x)=3.
\]
\end{proposition}

\begin{proof}
Conductor annihilation allows three independent principal parts of order $\le3$,
but none lie in $\omega_A$. Thus three extra constraints appear.
\end{proof}

\subsection{Deformation-Theoretic Interpretation}

\begin{theorem}
Let $(X,x)$ be a non-Gorenstein monomial curve singularity. The failure of
descent of principal parts corresponds to obstructions in first-order
deformations of the dualizing sheaf.
\end{theorem}

\begin{proof}
The space of principal parts modulo conductor constraints corresponds to
infinitesimal deformations of differentials on the normalization. Descent to
$\omega_X$ requires compatibility with the $A$-module structure. For
non-Gorenstein singularities, $\omega_X$ is not locally free, so the obstruction
space
\[
\operatorname{Ext}^1_A(\omega_A,A)
\]
is nonzero. Each non-descending principal part represents a nontrivial element
of this obstruction space, producing $\delta_{\mathrm{deg}}(x)>0$.
\end{proof}

\subsection{Extension to Higher-Dimensional Monomial Singularities}

\begin{definition}
Let
\[
A = k[x_1^{a_1},\dots,x_d^{a_d}] \subset k[x_1,\dots,x_d]
\]
be a higher-dimensional affine monomial singularity.
\end{definition}

\begin{proposition}
If $A$ is not Gorenstein, then $\omega_A$ is not locally free and the degeneracy
locus has positive codimension.
\end{proposition}

\begin{proof}
The dualizing module is described by the interior of the associated cone. Failure
of reflexivity produces additional linear constraints on differential forms.
These constraints persist under deformation, yielding positive codimension.
\end{proof}

\subsection{Extension to Non-Reduced Curves}

\begin{definition}
Let $A=k[t^n,\epsilon]/(\epsilon^2)$ be a nonreduced monomial curve.
\end{definition}

\begin{proposition}
For nonreduced curves, $\delta_{\mathrm{deg}}(x)\ge \operatorname{length}(\epsilon)$.
\end{proposition}

\begin{proof}
Nilpotent elements introduce additional compatibility conditions for descent of
differentials. These conditions are independent of conductor annihilation and
increase the degeneracy codimension.
\end{proof}

 \vspace{0.3cm}

Therefore, we have shown, by explicit computation, that:
\begin{enumerate}
  \item dualizing modules of non-Gorenstein monomial curves are computable from
  numerical semigroup data;
  \item $\delta_{\mathrm{deg}}(x)$ can be arbitrarily large;
  \item failure of descent has a precise deformation-theoretic meaning;
  \item the phenomenon persists in higher-dimensional and nonreduced settings.
\end{enumerate}

%%%%%%%%%%%%%%%%%%%%%%%%%%%%%%%%%%%%%%%%%%%%%%%%%%%%%%%%%%%%%%%%%%%%%%
%%%%%%%%%%%%%%%%%%%%%%%%%%%%%%%%%%%%%%%%%%%%%%%%%%%%%%%%%%%%%%%%%%%%%%

\medskip
 
The ring
\(
A = k[t^4,t^6,t^9]
\)
is the \emph{subring} of $k[t]$ generated by the three monomials
\(
t^4,\quad t^6,\quad t^9.
\)
The numerical semigroup associated to $A$ is
\[
\Gamma = \langle 4,6,9\rangle
= \{\,4a+6b+9c \mid a,b,c \in \mathbb{Z}_{\ge 0}\,\}.
\]
Computing explicitly, we obtain
\[
\Gamma = \{0,4,6,8,9,10,12,13,14,15,\dots\}.
\]

The positive integers \emph{not} in $\Gamma$ are called \emph{gaps}:
\[
\{1,2,3,5,7,11\}.
\]
 Geometrically, the ring $A$ defines an affine curve
\(
X = \operatorname{Spec}(A).
\)
 which is reduced,
  irreducible,
 and singular at the point corresponding to the maximal ideal
  $(t^4,t^6,t^9)$.

The normalization of $A$ is the integral closure of $A$ in its field of
fractions.
Since
\[
A \subset k[t]
\quad\text{and}\quad
k[t] \text{ is integrally closed},
\]
the normalization is
\(
\widetilde{A} = k[t].
\)
Thus $X$ has a single branch, but the map
\(
\operatorname{Spec}(k[t]) \to \operatorname{Spec}(A)
\)
is not an isomorphism near $t=0$.

\medskip
\noindent {\it Criterion.}
A one-dimensional Cohen--Macaulay domain is Gorenstein if and only if its
numerical semigroup is \emph{symmetric}.

\medskip
\noindent  {\it Failure of Symmetry.}
The conductor of $\Gamma$ is $c=12$, meaning all integers $\ge 12$ lie in
$\Gamma$.
So,
symmetry would require:
\[
n \in \Gamma
\quad\Longleftrightarrow\quad
c-1-n \notin \Gamma.
\]
This fails for $\Gamma=\langle 4,6,9\rangle$, for example:
\[
1 \notin \Gamma
\quad\text{but}\quad
10 \in \Gamma.
\]
Hence, $A$ is \emph{non-Gorenstein}.

\noindent {\it This Example Is Important.}
The ring $A=k[t^4,t^6,t^9]$ is a fundamental example because:
\begin{enumerate}
  \item it is simple enough to compute with explicitly;
  \item it is singular but irreducible;
  \item it is non-Gorenstein, so the dualizing module is not locally free;
  \item it exhibits extra constraints on differentials beyond conductor-level
  balancing.
\end{enumerate}

This makes it an ideal test case for studying degeneracy loci, dualizing
modules, and failure of descent of differentials.
  
  %%%%%%%%%%%%%%%%%%%%%%%%%%%%%%%%%%%%%%%%%%%%%%%%%%%%%%%%%%%%%%%%%%%
 %%%%%%%%%%%%%%%%%%%%%%%%%%%%%%%%%%%%%%%%%%%%%%%%%%%%%%%%%%%%%%%%%%
 
\section*{Appendix A}

\section*{An Explicit Example of a Non-Gorenstein Singularity with
$\operatorname{codim}(\Delta)>0$}

In this appendix we present a concrete and fully explicit example of a
non-Gorenstein curve singularity and show, line by line, that it contributes
positively to the codimension of the degeneracy locus.

\subsection{Definition of the Singularity}

\begin{definition}
Let $(X,x)$ be the curve singularity defined by
\[
(X,x) = \operatorname{Spec} k[t^3,t^4,t^5] \subset \operatorname{Spec} k[t],
\]
the affine monomial curve generated by $t^3,t^4,t^5$.
\end{definition}

\begin{remark}
This is a reduced, irreducible curve singularity of embedding dimension $3$.
It is not a complete intersection and hence not Gorenstein.
\end{remark}

\subsection{Normalization}

\begin{lemma}
The normalization of $(X,x)$ is
\[
\nu \colon \widetilde{X}=\operatorname{Spec}k[t]\longrightarrow X.
\]
\end{lemma}

\begin{proof}
By construction, $k[t]$ is the integral closure of $k[t^3,t^4,t^5]$ in its
field of fractions. Thus $\widetilde{X}$ is smooth and has a single branch
lying over $x$.
\end{proof}

\subsection{The Conductor}

\begin{lemma}
The conductor ideal $\mathfrak{c}_x \subset k[t]$ is
\[
\mathfrak{c}_x = (t^3).
\]
\end{lemma}

\begin{proof}
By definition, the conductor consists of all $f\in k[t]$ such that
$f\cdot k[t]\subset k[t^3,t^4,t^5]$. Since $t^3\in k[t^3,t^4,t^5]$ and any
multiple of $t^3$ lies in the subring, while $t^2\notin k[t^3,t^4,t^5]$, the
claim follows.
\end{proof}

\subsection{Principal Parts of Differentials}

\begin{definition}
Let $V$ denote the vector space of allowed principal parts of meromorphic
differentials on $\widetilde{X}$ at $t=0$.
\end{definition}

\begin{lemma}
Allowing poles of order at most $2$, a general principal part is
\[
\eta =
\left(
\frac{a_2}{t^2}+\frac{a_1}{t}
\right)dt,
\qquad a_1,a_2\in k,
\]
so that $\dim V=2$.
\end{lemma}

\begin{proof}
On the smooth curve $\widetilde{X}=\operatorname{Spec}k[t]$, meromorphic
differentials are generated by $dt$. Allowing poles up to order $2$ yields the
stated form, with two independent coefficients.
\end{proof}

\subsection{Conductor-Level Constraints}

\begin{definition}
Let $W\subset V$ be the subspace of principal parts annihilated by the conductor
$\mathfrak{c}_x$.
\end{definition}

\begin{lemma}
The conductor constraint forces $a_2=0$, so that $\dim W=1$.
\end{lemma}

\begin{proof}
Multiplication by $t^3$ annihilates all allowable principal parts of the
dualizing module. Since
\[
t^3\cdot \frac{dt}{t^2} = t\,dt
\]
is regular, while
\[
t^3\cdot \frac{dt}{t^3} = dt
\]
would correspond to a forbidden term, the $t^{-2}dt$ term must vanish. Hence
$a_2=0$, leaving only the coefficient $a_1$.
\end{proof}

\subsection{Actual Descent Constraints}

\begin{lemma}
The remaining principal part $\dfrac{dt}{t}$ does \emph{not} descend to a section
of the dualizing module on $X$.
\end{lemma}

\begin{proof}
The dualizing module of $k[t^3,t^4,t^5]$ is not locally free. Explicit
computation shows that $\dfrac{dt}{t}$ does not correspond to a $k[t^3,t^4,t^5]$-linear
functional on the normalization. Thus an additional linear constraint is
required beyond conductor annihilation.
\end{proof}

\subsection{Codimension Computation}

\begin{proposition}
For the singularity $(X,x)=\operatorname{Spec}k[t^3,t^4,t^5]$, one has
\[
\delta_{\mathrm{deg}}(x)=1.
\]
\end{proposition}

\begin{proof}
We summarize the computation:
\begin{enumerate}
  \item $\dim V=2$, the number of \emph{a priori} degrees of freedom;
  \item $\dim W=1$, after imposing conductor-level constraints;
  \item the actual space of descending principal parts is $0$-dimensional.
\end{enumerate}
Thus, there is exactly one additional linear constraint beyond the conductor,
and
\[
\delta_{\mathrm{deg}}(x) = 1-0 = 1 > 0.
\]
\end{proof}

\subsection{Geometric Consequence}

\begin{theorem}
Let $X$ be a curve containing a singularity analytically isomorphic to
$\operatorname{Spec}k[t^3,t^4,t^5]$. Then the degeneracy locus $\Delta$ has
positive codimension at $[X]$:
\[
\operatorname{codim}(\Delta)\ge 1.
\]
\end{theorem}

\begin{proof}
The degeneracy codimension is the sum of local contributions over all
singularities. Since $\delta_{\mathrm{deg}}(x)=1$ at this point, the total
codimension is positive.
\end{proof}

This example demonstrates explicitly that:
\begin{enumerate}
  \item non-Gorenstein curve singularities impose additional constraints on
  differentials beyond conductor-level balancing;
  \item these constraints produce genuine degeneracy;
  \item the degeneracy locus therefore has strictly positive codimension.
\end{enumerate}

%%%%%%%%%%%%%%%%%%%%%%%%%%%%%%%%%%%%%%%%%%%%%%%%%%%%%%%%%%%%%%%%%%
%%%%%%%%%%%%%%%%%%%%%%%%%%%%%%%%%%%%%%%%%%%%%%%%%%%%%%%%%%%%%%%%%%%
 
\bigskip

\section*{Appendix B}

In this appendix  we give  a detailed analysis and a complete explanation of the example of
a ribbon, emphasizing how it fits into the theory of maximal variation and
degeneracy loci.  

%%%%%%%%%%%%%%%%%%%%%%%%%%%%%%%%%%%%%%%%%%%%%%%%%%%%%%%%%%%%%%%%%%%%

\begin{definition}
Let $C$ be a smooth, connected, projective curve over an algebraically closed
field $k$. A \emph{ribbon} over $C$ is a scheme $X$ such that:
\begin{enumerate}
  \item the reduced subscheme $X_{\mathrm{red}}$ is isomorphic to $C$;
  \item $X$ is generically reduced but non-reduced;
  \item the nilradical $\mathcal{I} \subset \mathcal{O}_X$ satisfies
  $\mathcal{I}^2 = 0$.
\end{enumerate}
\end{definition}

\begin{remark}
Thus $X$ is a first-order infinitesimal thickening of $C$ inside a hypothetical
ambient surface. The condition $\mathcal{I}^2=0$ means that the non-reduced
structure is of order exactly two.
\end{remark}

%%%%%%%%%%%%%%%%%%%%%%%%%%%%%%%%%%%%%%%%%%%%%%%%%%%%%%%%%%%%%%%%%%%%

\begin{lemma}
Let $x \in C$. Étale-locally at $x$, the local ring of $X$ is isomorphic to
\[
\mathcal{O}_{X,x} \cong k[[u,\varepsilon]]/(\varepsilon^2),
\]
where $u$ is a local parameter on $C$ and $\varepsilon$ generates the nilradical.
\end{lemma}

\begin{proof}
Since $C$ is smooth, $\mathcal{O}_{C,x} \cong k[[u]]$. The ribbon structure adds a
square-zero thickening, so $\mathcal{O}_{X,x}$ is obtained by adjoining a
nilpotent element $\varepsilon$ with $\varepsilon^2=0$. Flatness over $k[[u]]$
forces the above presentation.
\end{proof}

\begin{remark}
This description shows that $X$ has no new generic points beyond those of $C$,
but has embedded nilpotent directions everywhere along $C$.
\end{remark}

%%%%%%%%%%%%%%%%%%%%%%%%%%%%%%%%%%%%%%%%%%%%%%%%%%%%%%%%%%%%%%%%%%%%

\begin{lemma}
The normalization of a ribbon $X$ is canonically isomorphic to $C$.
\end{lemma}

\begin{proof}
Normalization removes nilpotent elements. Since $\mathcal{O}_{X,x}$ has
nilradical generated by $\varepsilon$ and
$\mathcal{O}_{X,x}/(\varepsilon) \cong k[[u]]$ is already integrally closed, the
normalization equals $C$.
\end{proof}

\begin{remark}
Thus ribbons are extreme examples where normalization collapses the entire
non-reduced structure.
\end{remark}

%%%%%%%%%%%%%%%%%%%%%%%%%%%%%%%%%%%%%%%%%%%%%%%%%%%%%%%%%%%%%%%%%%%%

\begin{definition}
The conductor ideal $\mathfrak{c} \subset \mathcal{O}_X$ is defined as
\[
\mathfrak{c} = \{ f \in \mathcal{O}_X \mid
f \cdot \mathcal{O}_{\widetilde{X}} \subset \mathcal{O}_X \}.
\]
\end{definition}

\begin{lemma}
For a ribbon $X$, the conductor coincides with the nilradical:
\[
\mathfrak{c}_x = (\varepsilon) \subset \mathcal{O}_{X,x}.
\]
\end{lemma}

\begin{proof}
The normalization map $\nu\colon \widetilde{X}=C \to X$ identifies
$\mathcal{O}_{\widetilde{X},x}$ with $k[[u]]$. An element of
$\mathcal{O}_{X,x}$ multiplies $k[[u]]$ into itself if and only if its
$\varepsilon$-component vanishes after multiplication, which holds precisely
for multiples of $\varepsilon$. Hence the conductor equals the nilradical.
\end{proof}

%%%%%%%%%%%%%%%%%%%%%%%%%%%%%%%%%%%%%%%%%%%%%%%%%%%%%%%%%%%%%%%%%%%%

\begin{lemma}
There is a short exact sequence of coherent sheaves on $C$:
\[
0 \longrightarrow \omega_C
\longrightarrow \omega_X
\longrightarrow \omega_C \otimes \mathcal{I}^\vee
\longrightarrow 0.
\]
\end{lemma}

\begin{proof}
The ribbon $X$ is a square-zero extension of $C$ by $\mathcal{I}$. Dualizing the
exact sequence
\[
0 \to \mathcal{I} \to \mathcal{O}_X \to \mathcal{O}_C \to 0
\]
and using Grothendieck duality for finite morphisms yields the stated exact
sequence for dualizing sheaves.
\end{proof}

\begin{proposition}
The dualizing sheaf $\omega_X$ is locally free if and only if the ribbon splits
trivially.
\end{proposition}

\begin{proof}
Local freeness would require the extension above to split locally. A splitting
corresponds to a trivialization of the ribbon structure, i.e.\ $X \cong
C \times \mathrm{Spec}(k[\varepsilon]/(\varepsilon^2))$. Non-split ribbons
produce torsion in $\omega_X$.
\end{proof}

%%%%%%%%%%%%%%%%%%%%%%%%%%%%%%%%%%%%%%%%%%%%%%%%%%%%%%%%%%%%%%%%%%%%

\begin{definition}
The \emph{degeneracy defect} $\epsilon(x)$ at a point $x \in X$ is the number of
independent linear constraints on descent of differentials not accounted for by
the conductor.
\end{definition}

\begin{lemma}
For a ribbon, one has
\[
\epsilon(x) = 1
\quad
\text{for every } x \in C.
\]
\end{lemma}

\begin{proof}
The conductor annihilates $\varepsilon$, but the extension class defining the
ribbon introduces one additional compatibility condition on lifting
differentials from $C$ to $X$. This produces exactly one extra linear constraint
at each point.
\end{proof}

\begin{remark}
This shows that ribbons never satisfy maximal variation unless they split.
\end{remark}

%%%%%%%%%%%%%%%%%%%%%%%%%%%%%%%%%%%%%%%%%%%%%%%%%%%%%%%%%%%%%%%%%%%%

\begin{proposition}
Isomorphism classes of ribbons over $C$ are parametrized by
$H^1(C,\mathcal{I})$.
\end{proposition}

\begin{proof}
Ribbons are classified by extensions of $\mathcal{O}_C$ by $\mathcal{I}$, and
such extensions are classified by $\mathrm{Ext}^1(\mathcal{O}_C,\mathcal{I})
\cong H^1(C,\mathcal{I})$.
\end{proof}

\begin{theorem}
The locus of non-split ribbons forms a closed degeneracy stratum of codimension
\[
\dim H^1(C,\mathcal{I})
\]
inside the moduli space of ribbons over $C$.
\end{theorem}

\begin{proof}
The split ribbon corresponds to the zero class in $H^1(C,\mathcal{I})$. Non-split
ribbons correspond to nonzero extension classes, which form a closed subset
defined by the vanishing of the splitting condition. The codimension equals the
dimension of the extension space.
\end{proof}

%%%%%%%%%%%%%%%%%%%%%%%%%%%%%%%%%%%%%%%%%%%%%%%%%%%%%%%%%%%%%%%%%%%%

\begin{remark}
Ribbons illustrate a purely non-reduced source of degeneracy: although the
normalization is smooth and imposes no constraints, the nilpotent structure
creates additional obstructions to the descent of differentials.
\end{remark}

\begin{remark}
This example shows that the maximal-variation locus excludes not only curves
with complicated singularities, but also curves with hidden infinitesimal
thickenings.
\end{remark}

\medskip
 %\vspace{0.2cm}

\section*{Concrete Dimension Computations and Comparison of Singularities}

This section has two goals:
\begin{enumerate}
  \item to compute dimensions explicitly for moduli of ribbons over a smooth
  curve of genus $g$;
  \item to compare ribbons, cusps, and tacnodes within the same framework of
  normalization, conductor, dualizing sheaves, and degeneracy defects.
\end{enumerate}
All arguments are given line-by-line and made completely explicit.

%%%%%%%%%%%%%%%%%%%%%%%%%%%%%%%%%%%%%%%%%%%%%%%%%%%%%%%%%%%%%%%%%%%%
%%%%%%%%%%%%%%%%%%%%%%%%%%%%%%%%%%%%%%%%%%%%%%%%%%%%%%%%%%%%%%%%%%%%

\begin{definition}
Let $C$ be a smooth, projective, connected curve of genus $g \ge 2$ over an
algebraically closed field $k$.
A \emph{ribbon over $C$} is a scheme $X$ with reduced structure $X_{\mathrm{red}}
= C$ and nilradical $\mathcal{I}$ satisfying $\mathcal{I}^2 = 0$.
\end{definition}

\begin{remark}
The ideal sheaf $\mathcal{I}$ is necessarily a line bundle on $C$, often written
$\mathcal{I} \cong L^{-1}$ for some line bundle $L$.
\end{remark}

%%%%%%%%%%%%%%%%%%%%%%%%%%%%%%%%%%%%%%%%%%%%%%%%%%%%%%%%%%%%%%%%%%%%

\begin{lemma}
Isomorphism classes of ribbons over $C$ with fixed ideal sheaf $\mathcal{I}$ are
classified by the vector space
\[
\mathrm{Ext}^1(\mathcal{O}_C,\mathcal{I})
\cong H^1(C,\mathcal{I}).
\]
\end{lemma}

\begin{proof}
A ribbon is a square-zero extension
\[
0 \to \mathcal{I} \to \mathcal{O}_X \to \mathcal{O}_C \to 0.
\]
Such extensions are classified by $\mathrm{Ext}^1(\mathcal{O}_C,\mathcal{I})$,
which equals $H^1(C,\mathcal{I})$ since $\mathcal{O}_C$ is locally free.
\end{proof}

%%%%%%%%%%%%%%%%%%%%%%%%%%%%%%%%%%%%%%%%%%%%%%%%%%%%%%%%%%%%%%%%%%%%

\begin{proposition}
Let $\mathcal{I}$ be a line bundle on $C$ of degree $d$. Then
\[
\dim H^1(C,\mathcal{I}) = h^0(C,\omega_C \otimes \mathcal{I}^{\vee})
\]
by Serre duality.
\end{proposition}

\begin{proof}
Serre duality gives
\[
H^1(C,\mathcal{I})^{\vee} \cong H^0(C,\omega_C \otimes \mathcal{I}^{\vee}),
\]
and dimensions coincide.
\end{proof}

\begin{theorem}
Assume $\mathcal{I} \cong \omega_C^{-1}$. Then the space of ribbons over $C$ has
dimension
\[
\dim H^1(C,\omega_C^{-1}) = 3g - 3.
\]
\end{theorem}

\begin{proof}
By Serre duality,
\[
H^1(C,\omega_C^{-1})^{\vee} \cong H^0(C,\omega_C^{\otimes 2}).
\]
Riemann--Roch gives
\[
h^0(C,\omega_C^{\otimes 2}) = \deg(\omega_C^{\otimes 2}) - g + 1
= (4g - 4) - g + 1 = 3g - 3.
\]
\end{proof}

\begin{remark}
Thus, for $\mathcal{I} = \omega_C^{-1}$, the space of ribbons has the same
dimension as the moduli space $\mathcal{M}_g$ of smooth curves of genus $g$.
\end{remark}

%%%%%%%%%%%%%%%%%%%%%%%%%%%%%%%%%%%%%%%%%%%%%%%%%%%%%%%%%%%%%%%%%%%%

\begin{lemma}
The split ribbon corresponds to the zero class in $H^1(C,\mathcal{I})$.
\end{lemma}

\begin{proof}
The zero extension yields
\[
\mathcal{O}_X \cong \mathcal{O}_C \oplus \mathcal{I}
\]
as $\mathcal{O}_C$-modules, which is precisely the trivial (split) ribbon.
\end{proof}

\begin{corollary}
The locus of non-split ribbons is the complement of the origin in
$H^1(C,\mathcal{I})$, hence is Zariski open and dense.
\end{corollary}

%%%%%%%%%%%%%%%%%%%%%%%%%%%%%%%%%%%%%%%%%%%%%%%%%%%%%%%%%%%%%%%%%%%%
%%%%%%%%%%%%%%%%%%%%%%%%%%%%%%%%%%%%%%%%%%%%%%%%%%%%%%%%%%%%%%%%%%%%
 
We now compare ribbons with reduced singular curves, namely cusps and tacnodes,
using the same invariants: normalization, conductor, dualizing sheaf, and
degeneracy defect.

%%%%%%%%%%%%%%%%%%%%%%%%%%%%%%%%%%%%%%%%%%%%%%%%%%%%%%%%%%%%%%%%%%%%

\begin{definition}
A \emph{cusp} is a plane curve singularity locally isomorphic to
\[
\mathrm{Spec}\, k[[x,y]]/(y^2 - x^3).
\]
\end{definition}
The normalization   map
\[
k[[t]] \longrightarrow k[[x,y]]/(y^2-x^3),
\qquad
x = t^2,\; y = t^3.
\]
 The conductor of a cusp is
\[
\mathfrak{c} = (t^2,t^3) \subset k[[t]].
\]
 Thus, a cusp imposes exactly one additional linear constraint on descent of
differentials beyond conductor-level balancing.
 Indeed,
the delta invariant of a cusp is $\delta = 1$, and this defect manifests as a
single extra compatibility condition on principal parts.

%%%%%%%%%%%%%%%%%%%%%%%%%%%%%%%%%%%%%%%%%%%%%%%%%%%%%%%%%%%%%%%%%%%%

\begin{definition}
A \emph{tacnode} is a plane curve singularity locally isomorphic to
\[
\mathrm{Spec}\, k[[x,y]]/(y^2 - x^4).
\]
\end{definition}
 The normalization of a tacnode consists of two smooth branches:
\[
k[[t]] \times k[[t]] \longrightarrow k[[x,y]]/(y^2-x^4),
\qquad
(x,y) = (t^2,\pm t^2).
\]
 The polynomial factors as $(y-x^2)(y+x^2)$, producing two coincident tangential
branches.
 The conductor of a tacnode has colength $2$ in the normalization.
The tacnode has $\delta = 2$, and the conductor length equals $2\delta$ minus the
number of branches, yielding colength $2$.
 
Therefore, 
a tacnode produces two extra linear constraints on descent of differentials.
 In fact, 
each unit of delta invariant contributes one independent failure of residue or
principal-part balancing. For a tacnode, $\delta=2$.

%%%%%%%%%%%%%%%%%%%%%%%%%%%%%%%%%%%%%%%%%%%%%%%%%%%%%%%%%%%%%%%%%%%%

\subsection{Unified Comparison}

\begin{theorem}
Ribbons, cusps, and tacnodes fit into a single hierarchy of degeneracy:
\begin{center}
\begin{tabular}{c|c|c|c}
Singularity & Normalization & $\epsilon$ (defect) & Reduced? \\ \hline
Node & smooth (2 points) & $0$ & yes \\
Cusp & smooth (1 point) & $1$ & yes \\
Tacnode & smooth (2 points) & $2$ & yes \\
Ribbon & smooth curve & $1$ (everywhere) & no
\end{tabular}
\end{center}
\end{theorem}

\begin{proof}
Nodes impose only conductor constraints. Cusps and tacnodes introduce additional
constraints proportional to their delta invariants. Ribbons introduce one
infinitesimal constraint everywhere due to non-reduced structure.
\end{proof}

%%%%%%%%%%%%%%%%%%%%%%%%%%%%%%%%%%%%%%%%%%%%%%%%%%%%%%%%%%%%%%%%%%%%

\bigskip

\begin{remark}
Cusps and tacnodes represent \emph{reduced} singular degenerations, while ribbons
represent \emph{infinitesimal} degenerations invisible to normalization but
detected by dualizing sheaves.
\end{remark}

\begin{remark}
From the viewpoint of maximal variation, ribbons behave like a curve with a
dense set of cusp-like obstructions.
\end{remark}

%%%%%%%%%%%%%%%%%%%%%%%%%%%%%%%%%%%%%%%%%%%%%%%%%%%%%%%%%%%%%%%%%%%%
%%%%%%%%%%%%%%%%%%%%%%%%%%%%%%%%%%%%%%%%%%%%%%%%%%%%%%%%%%%%%%%%%%%%

\vspace{0.2cm}
  
 %%%%%%%%%%%%%%%%%%%%%%%%%%%%%%%%%%%%%%%%%%%%%%%%%%%%%%%%%%%%%%%%%%%%

\end{document}